\documentclass [11pt]{amsart}
\usepackage{amscd, amsmath,amsthm,amssymb, amsfonts, epsfig}
\usepackage{bbm}
\usepackage{graphicx}
\usepackage[utf8]{inputenc}
\usepackage{xcolor}
\usepackage{mathrsfs}  
\usepackage{esvect}
\usepackage{enumitem}
\usepackage{indentfirst}
\usepackage{caption}
\usepackage{subfigure,tikz}
\usepackage{tkz-euclide}
\usepackage{comment}
\usepackage{outlines}
\usepackage{multicol}
\usetikzlibrary{decorations.markings}

\usetikzlibrary{intersections,positioning, calc}

\newtheorem{theorem}{Theorem}[section]
\newtheorem{corollary}[theorem]{Corollary}
\newtheorem{lemma}[theorem]{Lemma}

\newtheorem{definition}[theorem]{Definition}
\newtheorem{proposition}[theorem]{Proposition}

\newtheorem{conjecture}{Conjecture}
\newtheorem{remark}[theorem]{Remark}

\begin{document}

\title[Uniqueness results for the critical catenoid]{Sufficient symmetry conditions for free boundary minimal annuli to be the critical catenoid}
\author{Dong-Hwi Seo}
\address{Research Institute for Natural Sciences, Hanyang University, 222 Wangsimni-ro, Seongdong-gu, Seoul, 04763, Republic of Korea}
\email{donghwi.seo26@gmail.com}


\begin{abstract}
    We first consider a uniqueness problem for embedded free boundary minimal annuli in the three-dimensional Euclidean unit half-ball. Then, we obtain symmetry properties for compact embedded free boundary minimal surfaces in the unit ball. Finally, we obtain several uniqueness results for the critical catenoid under symmetry conditions on the boundary. For example, we show that if an embedded free boundary minimal annulus whose boundary consists of two congruent components and has a reflection symmetry by a plane, then it is congruent to the critical catenoid.
\end{abstract}

\maketitle

\section{Introduction}
A free boundary minimal surface $M$ in the unit ball $\mathbb{B}^3\subset \mathbb{R}^3$ is a minimal surface in $\mathbb{B}^3$ that meets $\partial\mathbb{B}^3$ orthogonally along $\partial M$. It is a critical point of the area functional among all surfaces in $\mathbb{B}^3$ whose boundaries lie on $\partial \mathbb{B}^3$. It can be considered as a solution of the Plateau problem with free boundaries as in works of Courant \cite{Cou1938RPD, Cou1940EMS}. This topic was initially studied by Nitsche \cite{Nit85SPC} in 1985 and increasing attention has been paid to the topic after a new construction of free boundary minimal surfaces in a ball via extremal metrics of the normalized first Steklov eigenvalues of compact surfaces with boundary by the work of Fraser and Schoen \cite{FS11FSE, FS16SEB}, Petrides \cite{Pet19MSE}, and Matthiesen and Petrides \cite{MP20arXiv}. For a recent survey on this topic, see \cite{Li20FBM}.

One of the basic examples of free boundary minimal surfaces in $\mathbb{B}^3$ is the critical catenoid \cite{FS11FSE}. It is a portion of an appropriately scaled catenoid. More precisely, we consider the catenoid $\{x_1^2+x_2^2=\cosh^2{x_3}\}$ and all rays from the origin that are tangent to the catenoid. We take the subset of the catenoid that is the bounded component of the relative complement of the rays in the catenoid. Then, the critical catenoid $\Sigma_c$ is obtained by proper scaling so that the bounded component is properly embedded in $\mathbb{B}^3$ (see Figure \ref{fig:construction of the critical catenoid}). $\Sigma_c$ has many interesting properties. For example, it is homothetic to the maximally symmetric marginally stable piece of a catenoid \cite{BB2014VCC}.

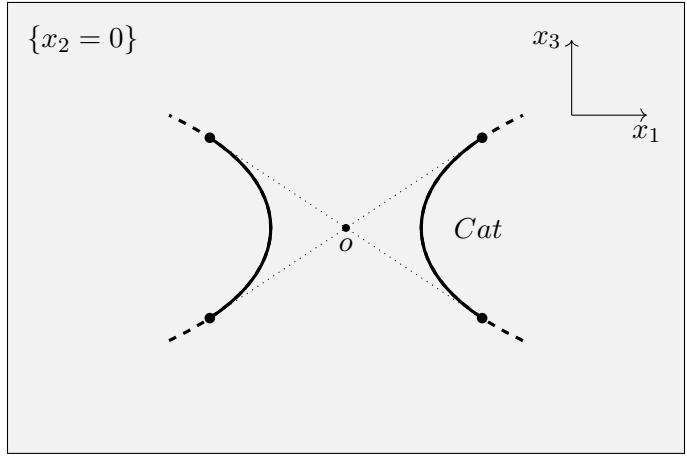
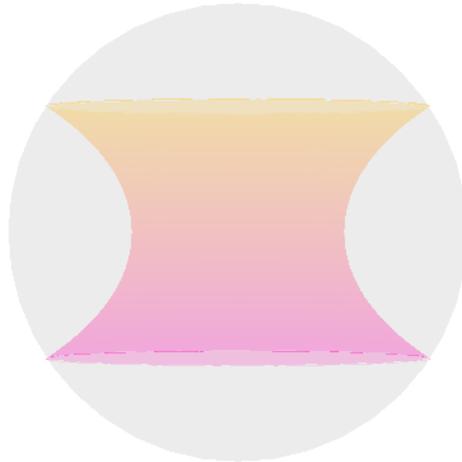
\begin{figure}
    \centering
    
        \subfigure[The intersection of $\{x_2=0\}$ and the catenoid $Cat:=\{x_1^2+x_2^2=\cosh^2{x_3}\}$. The dotted lines in the picture are the rays from the origin $o$ that are tangent to the intersection of $\{x_2=0\}$ and the catenoid $Cat$. The intersection consists of the thick curves and the dashed curves in the picture. The thick curves are the bounded components of the subset of $Cat\cap\{x_2=0\}$ that is the relative complement of the dotted rays in $Cat\cap \{x_2=0\}$.]{\label{subfig: catenery}
\begin{tikzpicture}
   \draw[fill=gray!10] (-4.5,3) -- (-4.5,-3) --(4.5,-3) --(4.5,3)--cycle;
   \draw [variable=\y,very thick,dashed, domain=-1.5:1.5] plot ({cosh(\y)}, {\y});
   \draw [variable=\y,very thick,dashed, domain=-1.5:1.5] plot ({-cosh(\y)}, {\y});
   \draw [variable=\y,very thick, domain=-1.19968:1.19968] plot ({cosh(\y)}, {\y});
   \draw [variable=\y,very thick, domain=-1.19968:1.19968] plot ({-cosh(\y)}, {\y});
   \draw[dotted, domain=-1.81017:1.81017] plot ({\x},{1.19968*\x/1.81017});
   \draw[dotted, domain=-1.81017:1.81017] plot ({\x},{-1.19968*\x/1.81017});

    \fill [black] (0,0) circle (1.5pt) node[below]{$o$};
     \fill [black] ({cosh(1.19968)}, {1.19968}) circle (2pt);
     \fill [black] ({cosh(1.19968)}, {-1.19968}) circle (2pt);
     \fill [black] ({-cosh(1.19968)}, {1.19968}) circle (2pt);
     \fill [black] ({-cosh(1.19968)}, {-1.19968}) circle (2pt);
     \node[right] at ({1.3}, {0}){$Cat$};
     \node at (-3.5,2.5){$\{x_2=0\}$};
   
   \draw[->] (3,1.5)--(3, 2.5) node[left]{$x_3$};
   \draw[->] (3, 1.5) -- (4, 1.5) node[below]{$x_1$};
   
\end{tikzpicture}
}
    \subfigure
    [Critical catenoid $\Sigma_c$]
    {\label{subfig: critical catenoid}  
    \includegraphics{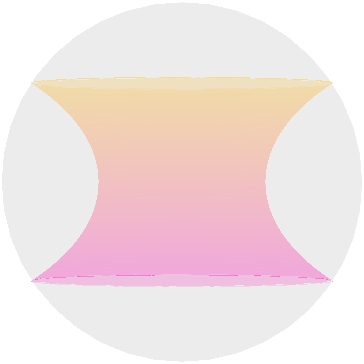}
}

    \caption{Construction of the critical catenoid. First, we take a component of the thick curves in Figure \ref{subfig: catenery}, which is a subset of $Cat$. Then, we obtain the surface of revolution by rotating the thick curve around the $x_3$-axis. If we scale it appropriately, we obtain the critical catenoid $\Sigma_c$ that is properly embedded in the unit ball $\mathbb{B}^3$ as in Figure \ref{subfig: critical catenoid}.}
    \label{fig:construction of the critical catenoid}
\end{figure}

We consider uniqueness problems for free boundary minimal surfaces in a ball in Euclidean space. Nitsche showed that every immersed free boundary minimal disk in $\mathbb{B}^3$ is congruent to the equatorial disk \cite{Nit85SPC}. In the same paper, Nitsche claimed the uniqueness conjecture, which is now well-known by Fraser and Li \cite{FL14CSE}, as follows.
\begin{conjecture}\label{conj: Nitsche}
Every embedded free boundary minimal annulus in $\mathbb{B}^3$ is congruent to the critical catenoid.
\end{conjecture}
We will denote by $\Sigma$ an embedded free boundary minimal annulus in $\mathbb{B}^3$.

Fraser and Schoen showed that if an immersed free boundary minimal annulus in the unit ball $\mathbb{B}^n, n\ge 3$, has the first Steklov eigenvalue 1, then it is congruent to the critical catenoid \cite{FS16SEB}. Using this result, McGrath showed that if $\Sigma$ is invariant under the reflections through three coordinate planes, $\Sigma$ is congruent to the critical catenoid \cite{McG18ACC}. Later, Kusner and McGrath extended this result by replacing the reflections through three coordinate planes with the antipodal map \cite{KM20arXiv}. For other rigidity results, see \cite{FS2020SRH,AN2021GTF,KL21FBM,Tra2021GMF,FHM2022FBM}.

In this paper, we obtain several uniqueness results for the critical catenoid (for full list of theorems, see Section \ref{sec: uniqueness of the critical catenoid}). For example,

\begin{theorem}\label{thm:congruent with one symmetry}If $\partial \Sigma$ consists of two congruent components and is invariant under the reflection through a plane, then $\Sigma$ is congruent to the critical catenoid.
\end{theorem}

To show Theorem \ref{thm:congruent with one symmetry}, we obtain some uniqueness results for the critical catenoid under symmetry conditions. In addition, we obtain symmetry properties for $\Sigma$. Then, observing the Gauss map of $\Sigma$ along a specific curve, we can show that $\Sigma$ is congruent to the critical catenoid. For details, see Section \ref{sec:congruent}.

Now, we consider our two main ingredients of the proof of Theorem \ref{thm:congruent with one symmetry}. First, we obtain the following uniqueness theorem for $\Sigma$ with symmetries by reflection planes (see Section \ref{sec:symmetry by planes}).
\begin{theorem} \label{thm:interior symmetry} If $\Sigma$ has one of the following symmetry conditions, then $\Sigma$ is congruent to the critical catenoid.
\begin{itemize} 
    \item $\Sigma$ is invariant under the reflections through two distinct planes. 
    \item $\Sigma$ is invariant under the reflection through a plane that does not meet $\partial \Sigma$.
\end{itemize}
\end{theorem}
The main idea of the proof of Theorem \ref{thm:interior symmetry} comes from the existence of a coordinate plane (a plane passing through the origin) with the following properties: This plane does not meet $\partial \Sigma$ and every plane perpendicular to this plane intersects $\partial\Sigma$ in at most four points (Proposition \ref{prop:existence of a coordinate plane}). This observation is derived from new geometric properties of strictly convex curves on $\mathbb{S}^2$ (see Section \ref{sec:geometry of curves}) and the concept of the $flux$ for minimal surfaces (see Section \ref{sec:properties of boundary}).

As a corollary, we have a uniqueness result for embedded free boundary minimal annuli in a half-ball as follows.
\begin{corollary}\label{cor:half-ball}
Let $\Sigma'$ be an embedded free boundary  minimal annulus in $\mathbb{B}^3\cap \{x_3\ge 0\}$. If one boundary component is contained in the open hemisphere and the other boundary component is contained in the equatorial disk, then $\Sigma'$ is congruent to the half of the critical catenoid, $\Sigma_c\cap \{x_3\ge0\}$. 
\end{corollary} 

The conclusion of Corollary \ref{cor:half-ball} is not expected to hold without the above assumption on the boundary, see Remark \ref{rmk:boundary condition}.\\

For the second ingredient of the proof of Theorem \ref{thm:congruent with one symmetry}, we obtain two symmetry properties: Lemma \ref{lemm:boundary symmetry} and the following symmetry principle. We will see Lemma \ref{lemm:boundary symmetry} later in Section \ref{sec:symmetry principle}. The symmetry principle implies that a symmetry of the boundary of a compact embedded free boundary minimal surface in $\mathbb{B}^3$ generates a global symmetry.
\begin{theorem}[Symmetry principle]\label{thm:symmetry principle}Let $M$ be a compact embedded free boundary minimal surface in $\mathbb{B}^3$. If $A\in O(3)$ maps a boundary component of $M$ into another boundary component (the two components can be identical), $M$ is invariant under $A$.
\end{theorem}

Together with this theorem and the result by Kusner and McGrath \cite{KM20arXiv}, which is previously mentioned, we obtain the following corollary.

\begin{corollary}\label{cor: antipodal symmetry on the boundary}
If $\partial \Sigma$ is invariant under the antipodal map, then $\Sigma$ is congruent to the critical catenoid.
\end{corollary}

In \cite{KL21FBM}, Kapouleas and Li showed that if $\Sigma$ has a rotationally symmetric boundary component, $\Sigma$ is congruent to the critical catenoid. Theorem \ref{thm:interior symmetry} and the symmetry principle give the following.
\begin{corollary} \label{cor:two plane with one component}If a component of $\partial \Sigma$ is invariant under the reflections through two distinct planes, then $\Sigma$ is congruent to the critical catenoid. 
\end{corollary}

In Section \ref{sec:background}, we review some of the standard facts on the Steklov eigenvalue problem. In Section \ref{sec:geometry of curves}, we prove some properties of a strictly convex curve on the 2-sphere which will be used in the next section. In Section \ref{sec:properties of boundary}, we prove several properties of $\partial\Sigma$. In Section \ref{sec:symmetry by planes}, we consider symmetry conditions by reflection planes for uniqueness theorems. In Section \ref{sec:symmetry principle}, we prove symmetry properties for a compact embedded free boundary  minimal surface in $\mathbb{B}^3$. In Section \ref{sec:congruent}, we prove Theorem \ref{thm:congruent with one symmetry}. Finally, we sum up our uniqueness results for the critical catenoid in Section \ref{sec: uniqueness of the critical catenoid}.

\section*{Acknowledgements} 
 
 The author would like to express his gratitude to Jaigyoung Choe, Kyeongsu Choi, Jaehoon Lee, Pablo Mira, Juncheol Pyo, and Eungbeom Yeon for valuable comments. The author also thanks José A. Gálvez for providing Lemma \ref{lemm: cone condition} that makes the proof of Proposition \ref{prop:good Gauss map} simpler than the previous proof.  The work was supported by the National Research Foundation of Korea (NRF) grant funded by the Korea Government (MSIT) (No. 2021R1C1C2005144).

\section{Background: Steklov eigenvalue problem}\label{sec:background}
Let $\Omega$ be a bounded domain with smooth boundary $\partial \Omega \neq \emptyset$ in a Riemannian manifold. Then, the Steklov eigenvalue problem, which is introduced by Steklov in 1902 \cite{Ste1902problemes}, is to find $\sigma\in \mathbb{R}$ for which there exists a function $u\in C^{\infty}(\Omega)$ satisfying
\begin{align} \label{problem}
\left\{
\begin{array}{rcll}
     \Delta u &=& 0&   \text{in   } \Omega  \\
     \frac{\partial u}{\partial \eta} &=& \sigma u &\text{on   } \partial \Omega
\end{array} \right.
,
\end{align}
where $\eta$ is the outward unit conormal vector of $\Omega$ along $\partial \Omega$. It is known that the eigenvalues $\sigma$ of this problem are discrete and form a sequence, $0=\sigma_0<\sigma_1\le \sigma_2\le\cdots\rightarrow \infty$ (see, for example, \cite{GP17SGS}). Note that constant functions are Steklov eigenfunctions with eigenvalue $\sigma_0=0$. In addition, for $i\ge 0$, we have
\begin{align} \label{variational characterization}
    \sigma_{i+1}(\Omega)= \inf_{f\in C^{\infty}(\partial \Omega)\setminus\{0\}}\left\{ \left.\frac{\int_{\Omega}|\nabla \hat{f}|^2}{\int_{\partial \Omega}f^2} \right| \int_{\partial \Omega}fu_k=0 \text{ for } k=0,\dots, i \right\},
\end{align}
where $\hat{f}\in C^{\infty}(\bar{\Omega})$ is the harmonic extension of $f$ and $u_k$ is a Steklov eigenfunction corresponding to $\sigma_k$.  

If $\Omega$ is a compact immersed free boundary minimal surface in $\mathbb{B}^3$, we have the following basic properties of Steklov eigenfunctions.
\begin{lemma} \label{lemm:Steklov eigenfunction}Let $M$ be a compact immersed free boundary minimal surface in $\mathbb{B}^3$. Then:
\begin{enumerate}[label=\arabic*)]
\item Any coordinate functions, $x_i, i=1,2,3$, are Steklov eigenfunctions of $M$ with eigenvalue 1. In particular, $\sigma_1(M)\le 1$. \label{lemm: eigenvalue 1}
\item Let $u$ be a first Steklov eigenfunction of $M$. Then, we have 
\begin{align}
    \int_{\partial M} u=0.
\end{align}
If $\sigma_1(M)<1$, we have 
\begin{align}
    \int_{\partial M} ux_i =0 \text{ for all }i=1,2,3.
\end{align} \label{lemm:orthogonality}
\end{enumerate}
\end{lemma}
\begin{proof}
\ref{lemm: eigenvalue 1} follows from \cite[Lemma 2.2]{FS11FSE} or \cite[Theorem 2.2]{Li20FBM}. Together with (\ref{variational characterization}), we obtain \ref{lemm:orthogonality}.
\end{proof}

\section{Geometry of curves on the two-sphere}\label{sec:geometry of curves}
In this section, we will study properties of strictly convex curves on $\mathbb{S}^2$.

\begin{definition}\label{def:strictly convex} A simple closed curve on $\mathbb{S}^2$ is said to be strictly convex if it intersects every great circle in $\mathbb{S}^2$ in at most two points.
\end{definition}

\begin{definition}
We define a function $F$ from a closed curve on $\mathbb{S}^2$ into $\mathbb{R}^3$ by
\begin{align}
    F(\mathcal{C})=\int_{\mathcal{C}}x,
\end{align}
where $\mathcal{C}$ is a closed curve on $\mathbb{S}^2$ and $x$ is the position vector of $\mathbb{R}^3$.
\end{definition}

\begin{proposition}\label{claim:flux}
A strictly convex closed curve $\mathcal{C}$ on $\mathbb{S}^2$ does not meet the coordinate plane perpendicular to $F(\mathcal{C})$.
\end{proposition}
\begin{proof}
Suppose it does. Then, there is a point $p\in \mathcal{C}$ such that $p$ is not contained in the open hemisphere centered at $F(\mathcal{C})$ (see Figure \ref{fig:Proposition 1(setting)}). 
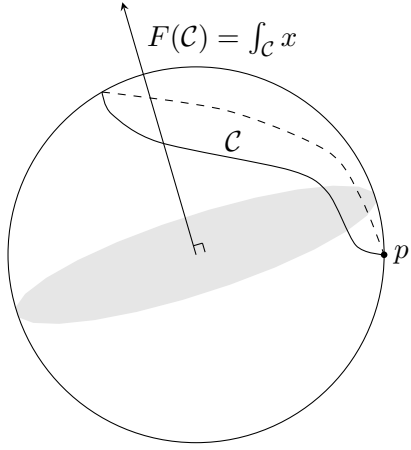
\begin{figure}
    \centering
    \begin{tikzpicture}[scale=2.5]
   \fill[domain=0:360, gray!20] plot ({cos(\x)*24/25-sin(\x)/4*7/25}, {cos(\x)*7/25-sin(\x)/4*24/25});

   \draw[-stealth ] (0,0)--({-7/25*1.4}, {24/25*1.4});
     \draw (0,0) circle (1cm);
     \draw plot [smooth] coordinates {(1,0)  (0.85,0.05) (0.7, {sqrt(11)/10}) (0.5,0.45) (-0.2, 0.6) (-0.45, 0.75) ({-1/2},{sqrt(3)/2})};

     \node[above] at (0.2, 0.5) {$\mathcal{C}$};

     \draw[dashed] plot [smooth] coordinates {({-1/2},{sqrt(3)/2}) (0.16, 0.7683) (0.7,0.55) ({0.09*sqrt(3)/2+0.8}, 0.3) (1,0)};

    \draw ({-7/25*0.05}, {24/25*0.05})--({-7/25*0.05+24/25*0.05}, {24/25*0.05+7/25*0.05});
    \draw ({24/25*0.05}, {7/25*0.05})--({-7/25*0.05+24/25*0.05}, {24/25*0.05+7/25*0.05});

     \node[right] at  ({-7/25*1.3+0.05}, {24/25*1.3-0.1}) {$F(\mathcal{C})=\int_{\mathcal{C}}x$};
     \fill [black] (1, 0) circle (0.5pt) node[right]{$p$};
\end{tikzpicture}
    \caption{Setting for the proof of Proposition \ref{claim:flux}. Let $\mathcal{C}$ be a closed strictly convex curve on $\mathbb{S}^2$. In the proof, we assume there is a point $p\in \mathcal{C}$ such that $p$ is not contained in the open hemisphere centered at $F(\mathcal{C}):=\int_{\mathcal{C}}x$, where $x$ is the position vector of $\mathbb{R}^3$.}
    \label{fig:Proposition 1(setting)}
\end{figure}
Let $p=(0,1,0), H_1=\{x_2<0\}\cap \mathbb{S}^2$, and $H_2=\{x_2>0\}\cap \mathbb{S}^2$. Then, by assumption, we have 
\begin{align}\label{assumption in proposition 1}
F(\mathcal{C})\in \bar{H}_1.    
\end{align}
We show that 
\begin{align}\label{dividing by length}
    L(\mathcal{C}\cap \bar{H}_1)<L(\mathcal{C}\cap H_2),
\end{align}
where $L(\mathcal{C}')$ is the length of a curve $\mathcal{C}'$ in $\mathbb{S}^2$. Since $\mathcal{C}$ is a strictly convex curve on $\mathbb{S}^2$, $\mathcal{C}\cap \{x_2=0\}$ is at most two points. If $\mathcal{C}\cap \{x_2=0\}$ is either the empty set or a point, $L(\mathcal{C}\cap \bar{H}_1)=0<L(\mathcal{C})=L(\mathcal{C}\cap H_2)$. Otherwise, $L(\mathcal{C}\cap H_2)\ge \pi$, which is the length of a great semicircle. Since $L(\mathcal{C})<2\pi$ (see \cite[Lemma 4]{HMOS08CCS}), we obtain (\ref{dividing by length}).

Using (\ref{dividing by length}), $\mathcal{C}$ is divided by $\mathcal{C}_1:= \mathcal{C}\cap \bar{H}_1$, $\mathcal{C}_2:= \mathcal{C}\cap \{x_2>c_1\}$ for some $c_1>0$ such that $L(\mathcal{C}_1)=L(\mathcal{C}_2)$, and $\mathcal{C}_3:= \mathcal{C}-(\mathcal{C}_1\cup \mathcal{C}_2)$. Then we construct subsets $\mathcal{C}_1'$ and $\mathcal{C}_2'$ of the great circle $G=\{x_3=0\}\cap \mathbb{S}^2$ by 
\begin{align}
    \mathcal{C}_1':= G\cap \{c_2<x_2\le 0\} \text{ for some } c_2<0 \text{ such that } L(\mathcal{C}_1')=L(\mathcal{C}_1)
\end{align}
and 
\begin{align}
    \mathcal{C}_2':= G\cap \{x_2>c_3\} \text{ for some } c_3>0 \text{ such that } L(\mathcal{C}_2')=L(\mathcal{C}_2).
\end{align}
See Figure \ref{fig:proposition 2 (constructions)} for the definitions of $\mathcal{C}_i, i=1,2,3,$ and $\mathcal{C}_j', j=1,2$. Note that $L(\mathcal{C}_1)=L(\mathcal{C}_1')= L(\mathcal{C}_2)=L(\mathcal{C}_2')$.
Then, we have $L(\{x_2>c\}\cap \mathcal{C}_2) \ge L(\{x_2>c\}\cap \mathcal{C}_2')$ for all $c\ge 0$. Thus, by the Fubini theorem, we have
\begin{align}
    \int_{\mathcal{C}_2}x_2 = \int_{0}^{\infty} L(\{x_2>c\} \cap \mathcal{C}_2) dc \ge \int^{\infty}_0 L(\{x_2>c\}\cap \mathcal{C}_2') dc = \int_{\mathcal{C}_2'}x_2.
\end{align}
Likewise, we have
\begin{align}
    \int_{\mathcal{C}_1}x_2 \ge \int_{\mathcal{C}_1'}x_2.
\end{align}
Thus,
\begin{align}
    \int_{\mathcal{C}} x_2  \ge \left(\int_{\mathcal{C}_1'}x_2  +\int_{\mathcal{C}_2'}x_2\right)+\int_{\mathcal{C}_3}x_2> 0,
\end{align}
contrary to (\ref{assumption in  proposition 1}).

\begin{figure}
    \centering
    \subfigure[$\mathcal{C}_1, \mathcal{C}_2$, and $\mathcal{C}_3$.]{
    \begin{tikzpicture}[scale=2.5]
   \fill[domain=0:360, gray!20] plot ({cos(\x)/4}, {sin(\x)});
     \fill[domain=90:270, gray!20] plot ({0.12*cos(\x)+0.8}, {3*sin(\x)/5});
     \fill[domain=-90:90, gray!20] plot ({0.09*cos(\x)+0.8}, {3*sin(\x)/5});
   
     \draw (0,0) circle (1cm);
     \draw plot [smooth] coordinates {(1,0)  (0.85,0.05) (0.7, {sqrt(11)/10}) (0.5,0.45) (-0.2, 0.6) (-0.45, 0.75) ({-1/2},{sqrt(3)/2})};
     
     \node[below] at (-0.45, 0.75) {$\mathcal{C}_1$};
     \node[below] at (0.2, 0.5) {$\mathcal{C}_3$};
     \node[left] at (0.85,0.05) {$\mathcal{C}_2$};

     \draw[dashed] plot [smooth] coordinates {({-1/2},{sqrt(3)/2}) (0.16, 0.7683) (0.7,0.55) ({0.09*sqrt(3)/2+0.8}, 0.3) (1,0)};
     
      \fill [black] (-0.2,0.6) circle (0.5pt);
     \fill [black] (0.16, 0.7683) circle (0.5pt);
     \draw[black] (0.7, {sqrt(11)/10}) circle (0.5pt);
     \draw[black] ({0.09*sqrt(3)/2+0.8}, 0.3) circle (0.5pt);
     \fill[black] (1,0) circle (0.5pt) node[right] {$p=(0,1,0)$};

     \node[above] at (0,1) {$\{x_2=0\}$};
     \node[right] at (0.77, 0.65) {$\{x_2=c_1\}$};
     
     \draw[->] (-1.5+3.5,0.8)--(-1+3.5,0.8);
     \node[below] at (-1+3.5, 0.8) {$x_2$};
     \draw[->] (-1.5+3.5,0.8)--(-1.5+3.5,1.3);
     \node[left] at (-1.5+3.5, 1.3) {$x_3$};
     \draw[->] (-1.5+3.5,0.8)--({-1.5-0.5/sqrt(2)+0.05+3.5},{0.8-0.5/sqrt(2)+0.05});
     \node at ({-1.5-0.5/sqrt(2)-0.05+3.5},{0.8-0.5/sqrt(2)-0.05}) {$x_1$};
\end{tikzpicture}
    }
    
    \subfigure[$\mathcal{C}_1'$ and $\mathcal{C}_2'$.]{
    \begin{tikzpicture}[scale=2.5]
   \fill[domain=0:360, gray!20] plot ({cos(\x)/4}, {sin(\x)});
     \fill[domain=0:360, gray!20] plot ({0.15*cos(\x)+0.6}, {0.8*sin(\x)});
     \fill[domain=0:360, gray!20] plot ({0.15*cos(\x)-0.6}, {0.8*sin(\x)});

     \draw (0,0) circle (1cm);
     \draw[dotted] plot [smooth] coordinates {(1,0)  (0.85,0.05) (0.7, {sqrt(11)/10}) (0.5,0.45) (-0.2, 0.6) (-0.45, 0.75) ({-1/2},{sqrt(3)/2})};
     
     \draw[domain=-62.87639:0] plot ({cos(\x)}, {sin(\x)/4});
     \draw[domain=0:41.6924, dashed] plot ({cos(\x)}, {sin(\x)/4});
     \draw[domain=75.9627:{180-62.87639}, dashed] plot ({cos(\x)}, {sin(\x)/4});
    \draw[domain={180+41.6924}:255.9627] plot ({cos(\x)}, {sin(\x)/4});
     \node[above] at (0.2, 0.5) {$\mathcal{C}$};

     \draw[dotted] plot [smooth] coordinates {({-1/2},{sqrt(3)/2}) (0.16, 0.7683) (0.7,0.55) ({0.09*sqrt(3)/2+0.8}, 0.3) (1,0)};
     
      \fill [black] (-0.2,0.6) circle (0.5pt);
     \fill [black] (0.16, 0.7683) circle (0.5pt);
     \draw[black] (0.7, {sqrt(11)/10}) circle (0.5pt);
     \draw[black] ({0.09*sqrt(3)/2+0.8}, 0.3) circle (0.5pt);
     \fill[black] (1,0) circle (0.5pt) node[right] {$p=(0,1,0)$};
     \draw[black] ({cos(-62.87639)}, {sin(-62.87639)/4}) circle (0.5pt);
    \draw[black] ({cos(41.6924)}, {sin(41.6924)/4}) circle (0.5pt);
    \fill [black] ({cos(75.9627)}, {sin(75.9627)/4}) circle (0.5pt);
    \fill [black] ({cos(255.9627)}, {sin(255.9627)/4}) circle (0.5pt);
    \draw[black] ({cos(180-62.87639)}, {sin(180-62.87639)/4}) circle (0.5pt);
    \draw[black] ({cos(180+41.6924)}, {sin(180+41.6924)/4}) circle (0.5pt);
    
    \draw[->] ({-0.09}, {0.35*sin(96.543155)/4})--({cos(96.543155)}, {sin(96.543155)/4});
    \draw[->] ({0.3*cos(238.82755)}, {0.3*sin(238.82755)/4})--({cos(238.82755)}, {sin(238.82755)/4});
     
     \node[above] at (0,1) {$\{x_2=0\}$};
     \node[right] at (0.6, 0.9) {$\{x_2=c_3\}$};
     \node[left] at (-0.6,0.9) {$\{x_2=c_2\}$};
     \node[above] at ({cos(-40)}, {sin(-40)/4})  {$\mathcal{C}_2'$};
     \node at (-0.08, 0) {$\mathcal{C}_1'$};

      \draw[->] (-1.5+3.5,0.8)--(-1+3.5,0.8);
     \node[below] at (-1+3.5, 0.8) {$x_2$};
     \draw[->] (-1.5+3.5,0.8)--(-1.5+3.5,1.3);
     \node[left] at (-1.5+3.5, 1.3) {$x_3$};
     \draw[->] (-1.5+3.5,0.8)--({-1.5-0.5/sqrt(2)+0.05+3.5},{0.8-0.5/sqrt(2)+0.05});
     \node at ({-1.5-0.5/sqrt(2)-0.05+3.5},{0.8-0.5/sqrt(2)-0.05}) {$x_1$};
    
\end{tikzpicture}
    
    }
    
    \caption{Definitions of $\mathcal{C}_i, i=1,2,3,$ and $\mathcal{C}_j', j=1,2$. Let $\mathcal{C}$ be a closed strictly convex curve on $\mathbb{S}^2$ and assume $p=(0,1,0)\in \mathcal{C}$. Shaded regions in the pictures are intersections of $\mathbb{B}^3$ and planes $\{x_2=c\}$ for $c\in \mathbb{R}$. (a) $\mathcal{C}_i, i=1,2,3,$ are subsets of $\mathcal{C}$. Let $\mathcal{C}_1=\mathcal{C}\cap\{x_2\le 0\}$. $\mathcal{C}_2$ is defined by $\mathcal{C}\cap\{x_2> c_1\}$ such that $L(\mathcal{C}_1)=L(\mathcal{C}_2)$. Then, let $\mathcal{C}_3=\mathcal{C}-(\mathcal{C}_1\cup \mathcal{C}_2)$. (b) $\mathcal{C}_j', j=1,2,$ are subsets of the great circle $G$ in $\{x_3=0\}$. $\mathcal{C}_1'$ is defined by $G\cap\{c_2<x_2\le 0\}$ such that $L(\mathcal{C}_1')=L(\mathcal{C}_1)$. $\mathcal{C}_2'$ is defined by $G\cap \{x_2>c_3\}$ such that $L(\mathcal{C}_2')=L(\mathcal{C}_2)$.}
    \label{fig:proposition 2 (constructions)}
\end{figure}
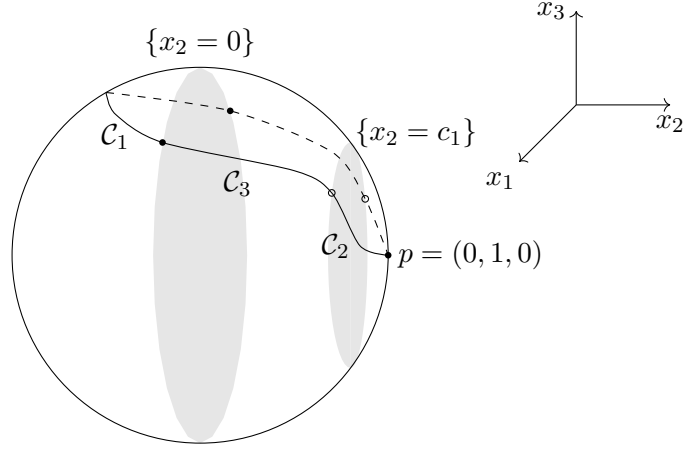
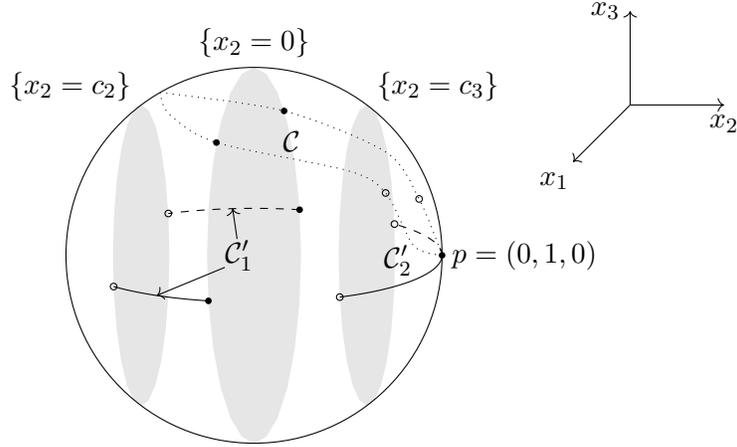
\end{proof}

\begin{remark}
Proposition \ref{claim:flux} provides a constructive proof of the existence of a great circle that does not meet a given strictly convex curve on $\mathbb{S}^2$ (see \cite[Lemma 4 and Theorem 2]{HMOS08CCS} for an existence proof via integral geometry).
\end{remark}

The following argument and Proposition \ref{prop:interior} are well-known, but we give a proof for the convenience of the reader. A strictly convex spherical set $D$ is a closed subset of $\mathbb{S}^2$ such that it lies in an open hemisphere and the shortest geodesic segments between any two distinct points in $D$ are contained in $D$ (see, for instance, Santal\'{o} \cite{San1946CRN}). By Proposition \ref{claim:flux}, any strictly convex curve $\mathcal{C}$ on $\mathbb{S}^2$ bounds the exactly one component of $\mathbb{S}^2\setminus \mathcal{C}$ that is contained in an open hemisphere. Then, we can show that the closure of such component in $\mathbb{S}^2$, say $\mathcal{D}$, is a strictly convex spherical set. Indeed, if it is not true, we can find $p,q\in \mathcal{C}$ and a great circle $G$ connecting $p$ and $q$ such that any point in the shortest arc in $G$ connecting $p$ and $q$ does not contained in $\text{Int}(\mathcal{D})$. Note that the length of the arc is less than $\pi$ because $\mathcal{C}$ is contained in an open hemisphere by Proposition \ref{claim:flux}. Since $G\cap \mathcal{C}=\{p, q\}$, $\mathcal{D}$ contains the longest arc connecting $p$ and $q$ in $G$, so it cannot be contained in any closed hemispheres, a contradiction. Thus, we showed the first statement in the following proposition.

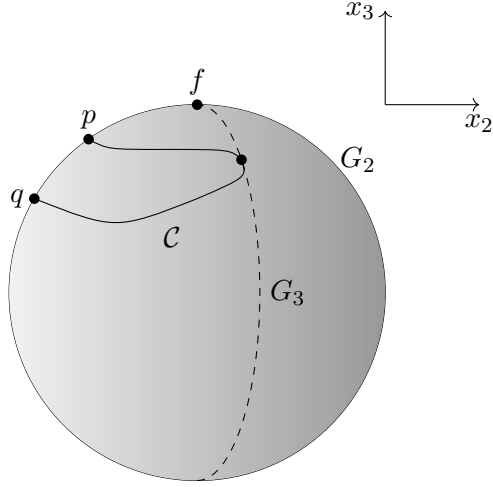
\begin{figure}
    \centering
    \begin{tikzpicture}[scale=2.5]

     \draw (0,0) circle (1cm);
      \shade[left color=gray!10,right color=gray!80] (0,0) circle (1cm);
      
      \draw[dashed] (0,-1) arc(-90:90:{1/3} and 1);
      
      \draw ({-1/sqrt(3)}, {sqrt(2)/sqrt(3)}).. controls ({(-1/sqrt(3))+0.1}, {(sqrt(2)/sqrt(3)+1/sqrt(2))/2}) .. ({(-1/sqrt(3)+1/sqrt(18))/2}, {(sqrt(2)/sqrt(3)+1/sqrt(2))/2}).. controls ({(-1/sqrt(3)+1/sqrt(18))/2+0.35}, {(sqrt(2)/sqrt(3)+1/sqrt(2))/2})..
      ({1/sqrt(18)}, {1/sqrt(2)}).. controls ({1.13/sqrt(18)},{0.87/sqrt(2)})..(0,{1/2})..controls ({-sqrt(3)/4},{0.33}).. ({-sqrt(3)/2},{1/2}) ;

    \fill [black] (0,1) circle (0.8 pt) node[above]{$f$};
    \fill [black] ({1/sqrt(18)}, {1/sqrt(2)}) circle ( 0.8pt);
    \fill[black] ({-1/sqrt(3)},{sqrt(2)/sqrt(3)}) circle (0.8 pt) node[above]{$p$};
    \fill[black] ({-sqrt(3)/2},{1/2}) circle (0.8 pt) node[left]{$q$};
    \node[below]at ({-sqrt(3)/4+0.3},{0.4}) {$\mathcal{C}$};
    \node [right] at ({1/3},0) {$G_3$};
    \node [right] at ({1/sqrt(2)}, {1/sqrt(2)}) {$G_2$};
    
    \draw[->] (1,1)--(1, 1.5) node[left]{$x_3$};
   \draw[->] (1, 1) -- (1.5, 1) node[below]{$x_2$};
     
\end{tikzpicture}
    \caption{The existence of a great circle $G_3$ that intersects a component of $\mathcal{C}\setminus G_2$ at exactly one point in the proof of Proposition \ref{prop:interior}. In the proof, we assumed $f$ does not lie in the geodesic segment of $\mathbb{S}^2$ connecting $p, q\in \mathcal{C}$. We take the great circle $G_2$ that passes through $f, p$, and $q$. A component of $\mathbb{S}^2\setminus G_2$ is represented by the gray region in the picture. $\mathcal{C}\setminus G_2$ has two components by the strictly convexity of $\mathcal{C}$. Note that $\mathcal{C}\cap G_2=\{p,q\}$ and $p$ and $q$ are contained in the open hemisphere centered at $f$ (Proposition \ref{claim:flux}). Then, we can take a great circle $G_3$ that intersects a component $\mathcal{C}\setminus G_2$ at exactly one point.}
    \label{fig:Proposition 2}
\end{figure}

\begin{proposition} \label{prop:interior}
A strictly convex closed curve $\mathcal{C}$ bounds a strictly convex spherical set $\mathcal{D}$ and
\begin{align}\label{f}
    f:= \frac{F(\mathcal{C})}{\lVert F(\mathcal{C})\rVert} \in \text{Int}(\mathcal{D}).
\end{align}
\end{proposition}
\begin{proof}
Suppose $f \in \mathcal{C}$. By the strictly convexity of $\mathcal{C}$, we can find a great circle $G_1$ passing through $f$ such that $\mathcal{C}$ is contained in the one of the closed hemispheres with boundary $G_1$. Let $G_1:= \{v \cdot x=0\}\cap \mathbb{S}^2$ for some nonzero vector $v\in \mathbb{R}^3$, and we may assume $\mathcal{C}$ is contained in $\{v\cdot x\ge 0\}$. Note that $\mathcal{C}\neq G_1$. Then, $v\cdot \int_{(\partial\Sigma)_i}x ds>0$, so $f$ is contained in $\{v\cdot x> 0\}$. This contradicts $f\in G_1$. Thus, $f \notin \mathcal{C}$.

We assert that if there is a geodesic segment of $\mathbb{S}^2$ connecting $f$ and two points $p, q\in\mathcal{C}$, then it has endpoints $p$ and $q$. Suppose not. We may assume the endpoints are $f$ and $q$. By the strictly convexity of $\mathcal{C}$, the great circle $G_2$ passing through $f, p,$ and $q$, dissects $\mathcal{C}$ into two components. We may assume $f=(0,0,1)$, $G_2=\{x_1=0\}\cap \mathbb{S}^2$, and $p\in G_2 \cap \{x_2<0\}$. By Proposition \ref{claim:flux}, $\mathcal{C}\subset \{x_3>0\}$. Then, $p, q\in \{x_2<0\}\cap \{x_3>0\}\cap G_2$. In addition, by the strictly convexity of $\mathcal{C}$, $\mathcal{C}$ does not meet $G_2\cap \{x_2>0\}$. Thus, for a fixed component of $\mathcal{C}\setminus G_2$, there exists a great circle $G_3$ passing through $f$ such that they meet at exactly one point (See Figure \ref{fig:Proposition 2}). Furthermore, such circumstance occurs if and only if they meet at the point tangentially. If $\mathcal{C}$ is contained in the one of closed hemispheres with boundary $G_3$, we can obtain a contradiction in a similar way to the previous paragraph. Otherwise, $G_3$ intersects $\mathcal{C}$ in at least three points. This also contradicts Lemma \ref{lemm:byproduct of TPP}. Thus, we proved the assertion, and it clearly implies that $f\in \text{Int}(\mathcal{D})$.

\end{proof}

\begin{figure}
    \centering
    \includegraphics[width=\linewidth]{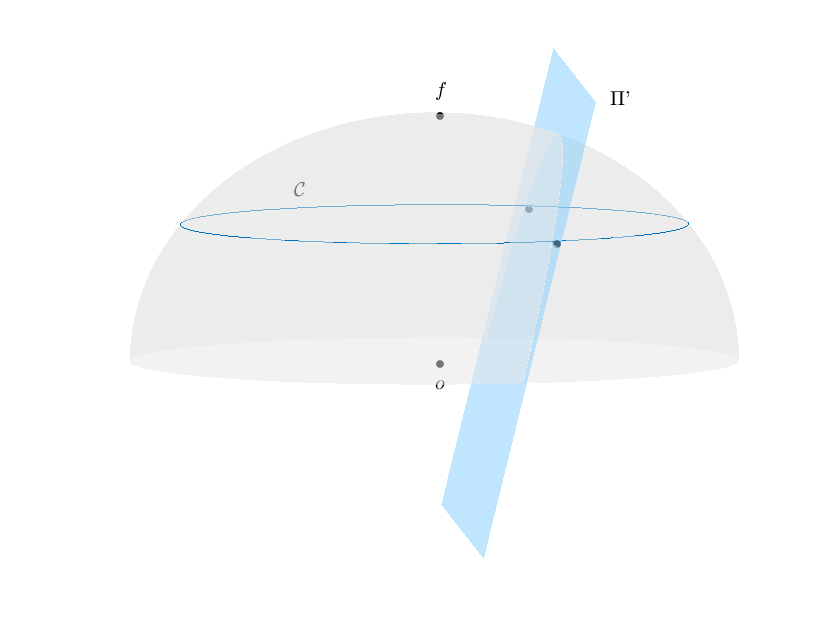}
    \caption{An illustration of Proposition $\ref{prop:at_most_two}$. We consider the open hemisphere centered at $f$ and a strictly convex curve $\mathcal{C}$. Let $\Pi'$ be a plane that bounds a strictly convex spherical cap of $\mathbb{S}^2$ whose center lies in the complement of the open hemisphere in $\mathbb{S}^2$. Note that $f$ and the origin $o$ lie on the same side of $\Pi'$. Then, Proposition \ref{prop:at_most_two} says $\Pi'$ intersects $\mathcal{C}$ in at most two points as in the picture.}
    \label{fig:proposition 3}
\end{figure}

In the following proposition, we will consider a class of planes that intersect a strictly convex curve $\mathcal{C}$ in at most two points (see Figure \ref{fig:proposition 3}). Note that planes in the class to be considered may not pass through the origin. 

\begin{proposition}\label{prop:at_most_two}
Let $H$ be the open hemisphere centered at $f$.
Let $\Pi'$ be a plane that bounds a strictly convex spherical cap of $\mathbb{S}^2$ whose center lies in $\mathbb{S}^2\setminus H$. Then, $\Pi'$ intersects $\mathcal{C}$ in at most two points.
\end{proposition}
\begin{proof} Suppose not.
Let $C'=\Pi'\cap \mathbb{S}^2$ and $p_1, p_2, p_3 \in \Pi'\cap \mathcal{C}$. We may assume that $f=(0,0,1)$ and the center of the strictly convex spherical cap is $o'=(0,a_2,a_3)$ with $a_2>0$ and $a_3\le 0$. Additionally, we may assume that $p_2$ is contained in the shortest arc in $C'$ connecting $p_1$ and $p_3$. It is clear that $\Pi'$ does not pass through $f$ by assumption. Then the great circle passing through $f$ and $p_2$ intersects $\mathbb{S}^2\cap\{x_2>0\}\cap\{x_3=0\}$ in a point, say $q_1\in \mathbb{S}^2$. Note that the geodesic segment from $p_2$ to $q_1$ is contained in the strictly convex spherical cap. Thus, the segment intersects the geodesic segment between $p_1$ and $p_3$ at a point, say $q_2$. Then, $q_2\in \text{Int}(\mathcal{D})$ because it is contained in the geodesic segment between $p_1, p_3\in \mathcal{C}=\partial \mathcal{D}$ (see Proposition \ref{prop:interior}). Thus, $p_2\in \partial \mathcal{D}$ is contained in the geodesic segment between $q_2$ and $f \in \text{Int}(\mathcal{D})$ (see Proposition \ref{prop:interior}), which is impossible.
\end{proof}

\section{Properties of boundary}\label{sec:properties of boundary}
In this section, we prove three propositions about $\partial \Sigma$, where $\Sigma$ is an embedded free boundary minimal annulus in $\mathbb{B}^3$ as in the introduction section. Before introducing the propositions, we first recall the two-piece property obtained by Lima and Menezes \cite{LM21TPP}. 
\begin{theorem}[Two-piece property, {\cite[Theorem A]{LM21TPP}}]\label{thm:two-piece property} A coordinate plane dissects a compact embedded free boundary minimal surface in $\mathbb{B}^3$ into exactly two components, excluding equatorial disks.
\end{theorem}
From now on, let $M_0$ denote a compact embedded free boundary minimal surface of genus zero with at least two boundary components in $\mathbb{B}^3$. In addition, let 
\begin{align}
(\partial M_0)_1,\dots, (\partial M_0)_b, \quad b\ge 2,
\end{align} 
be the boundary components of $M_0$.
 
Kusner and McGrath obtained the strictly convexity of the boundary components of $M_0$ \cite{KM20arXiv}. We give next a detailed proof of this fact, since we will use some arguments of this proof later on.

\begin{lemma}[{\cite[Corollary 4.2]{KM20arXiv}}] \label{lemm:byproduct of TPP} 
Each $(\partial M_0)_i, i=1,\dots, b,$ is a strictly convex curve on $\mathbb{S}^2$ (see Definition \ref{def:strictly convex} for a strictly convex curve on $\mathbb{S}^2$).
\end{lemma}
\begin{proof}
Let $\Pi$ be a coordinate plane. $(\partial M_0)_i\cap\Pi$ consists of transverse points for $\Pi$ and tangential points for $\Pi$. Since the linear function of $\Pi$ is a Steklov eigenfunction of $M_0$, $M_0\cap \Pi$ around a tangential point of $(\partial M_0)_i$ has at least two distinct arcs from the point that meet the boundary transversely (see the proof of \cite[Theorem 2.3]{FS16SEB}). In addition, the zero set of the linear function consists of a finite number of immersed circles and arcs between pairs of boundary components of $M_0$ (\cite[Theorem 2.5]{Che76ENS} or the proof of \cite[Theorem 2.3]{FS16SEB}). Now we show that $|(\partial M_0)_i\cap\Pi_1|\le 2$ by step-by-step. Theorem \ref{thm:two-piece property} (two-piece property) is a main tool in the proof. 

\textbf{Step 1.} We define a graph $G$ whose vertices are the components of $\partial M_0$ containing points in $\Pi$ and whose edges are the arcs in $\Sigma \cap \Pi$ whose endpoints are contained in $\partial M_0$. These arcs can be immersed circles containing a point in $\partial M_0$. In addition, every vertex has nonzero degree by definition.

\textbf{Step 2.} We show that the degree of each vertex in $G$ is at least two. If $(\partial M_0)_i$ has a tangential point, the vertex corresponding to the component has degree at least two by the argument before Step 1. If $(\partial M_0)_i$ has only transverse points, we can see that considering the sign of the linear function of $\Pi$ the boundary component has an even number of transverse points. Thus, we proved our claim in this step.

\textbf{Step 3.} We show that $G$ contains exactly one cycle. Moreover, $\partial M_0\cap \Pi$ does not contain loops that does not meet $\partial M_0$. By Step 2, $G$ has a cycle. Since $M_0$ has genus zero, the existence of a cycle in $G$ implies $M_0\setminus \Pi$ has at least two components. Moreover, if $G$ has at least two distinct cycles then $M_0\setminus \Pi$ has at least three components, which contradicts Theorem \ref{thm:two-piece property}. In particular, $\partial M_0\cap \Pi$ does not contain loops that does not meet $\partial M_0$. Thus, we proved our claim in this step. 

\textbf{Step 4.} We show that if $(\partial M_0)_i$ has a tangential point, then this point is the only point of $(\partial M_0)_i\cap \Pi$. Suppose not. Then, the vertex $v_1$ in $G$ corresponding to $(\partial M_0)_i$ has degree at least three. By Step 2, we can take a shortest cycle $C$ that contains $v_1$. Then, there exists an edge $e=(v_1 v_2)$ that is not contained in $C$. If there is a path from $v_2$ that does not contain $e$ contains a vertex in $C$, we have two distinct cycles in $G$, a contradiction with Step 3. Otherwise, since each vertex in $V(G)\setminus V(C)\setminus \{v_2\}$ has degree at least two by Step 2, we have another cycle in $G$, a contradiction with Step 3. Here, $V(G)$ and $V(C)$ are the sets of all vertices in $G$ and $C$, respectively. Thus, $(\partial M_0)_i\cap \Pi$ is exactly the tangential point that we assumed.

\textbf{Step 5.} We show that if $(\partial M_0)_i$ has only transverse points, then it has exactly two transverse points. Considering the sign of the linear function of $\Pi$, the number of transverse points in $(\partial M_0)_i$ is even. If $(\partial M_0)_i$ has at least four transverse points, a contradiction follows from Fraser and Schoen's argument in \cite[Proposition 8.1]{FS16SEB} with replacing the condition $\sigma_1=1$ by the two-piece property (Theorem \ref{thm:two-piece property}). Note that the fact that $M_0$ has genus zero is used in the argument. Thus, $(\partial M_0)_i$ contains exactly two transverse points. 

From Step 4 and 5, we have $|(\partial M_0)_i \cap \Pi| \le 2$, which completes the proof.
\end{proof}

\begin{corollary}\label{cor:the intersection points in a boundary component}
Each $(\partial M_0)_i, i=1, \dots, b,$ intersects a coordinate plane in at most two points. If they meet at exactly two points, $\Sigma$ meets the coordinate plane transversely at the points. If they meet at exactly one point, $\Sigma$ meets the coordinate plane tangentially at the point.
\end{corollary}
\begin{proof}
It follows from the proof of Lemma \ref{lemm:byproduct of TPP} (Step 4 and 5).
\end{proof}

\begin{remark}\label{rmk:at most two by first eigenfunction}For the proofs of Lemma \ref{lemm:byproduct of TPP} and Corollary \ref{cor:the intersection points in a boundary component}, we considered linear functions as Steklov eigenfunctions of $M_0$ that have exactly two nodal domains (see Definition \ref{def:nodal domain} for a nodal domain). Thus, by a similar argument, each of any first Steklov eigenfunctions of $M_0$ has at most two zeros in $(\partial M_0)_i$ for all $i$.
\end{remark}

\begin{corollary}[Transversality] \label{cor: transversality} If $M_0$ intersects a coordinate plane at a point in the interior of $M_0$, then, they meet transversely at the point. 
\end{corollary}
\begin{proof}

Suppose a coordinate plane $\Pi$ intersects $\text{Int}(M_0)$ tangentially. Then, $\Pi\cap \text{Int}(M_0)$ is isolated and we denote it by $\{v_1',\dots, v_k'\}$. We define a graph $G'$ whose vertices are $\{v_1', \dots, v_k'\}$ and the components of $\partial M_0$ containing points in $\Pi$ and whose edges are the arcs in $\Sigma \cap \Pi$ whose endpoints are contained in $\partial M_0\cup \{v_1',\dots, v_k'\}$. By \cite[Theorem 6.18]{CM11CMS}, $\text{deg}(v_i')=2m_i$ for $i=1,\dots, k$, where $m_i>1$ for all $i$. Since every vertex of $G'$ has degree at least two, as in Step 4 in the proof of Lemma \ref{lemm:byproduct of TPP}, we have two distinct cycles in $G'$. It contradicts the two-piece property with the fact that $M_0$ has genus zero (see Step 3 in the proof of Lemma \ref{lemm:byproduct of TPP}).    

\end{proof}

Let $\partial \Sigma = (\partial \Sigma)_1 \cup (\partial \Sigma)_2$, where $(\partial \Sigma)_i, i=1,2$, are distinct boundary components of $\Sigma$. For each $(\partial \Sigma)_i$, we define the following function, which is motivated from the $flux$ of a minimal surface.
 
\begin{definition}\label{def:f_i}
We define $f_i\in \mathbb{S}^2, i=1,2$, by
\begin{align}
    f_i:=\frac{F((\partial \Sigma)_i)}{\lVert F((\partial \Sigma)_i)\rVert}.
\end{align}
\end{definition}

The $flux$ is used to define a map from one-dimensional homology classes of a minimal surface in $\mathbb{R}^3$ into $\mathbb{R}^3$. It is defined by the integral of an oriented unit conormal vector of the surface along a representative closed curve (see, for example, \cite[Corollary 1.8]{CM11CMS}). Note that $(\partial \Sigma)_1$ and $(\partial \Sigma)_2$ are homologous and their position vectors become the outward unit conormal vectors of $\Sigma$ along them. Thus, we have
\begin{align}\label{sum_of_flux}
    F((\partial \Sigma)_1)+F((\partial \Sigma)_2)=0.
\end{align}
Then, we have
\begin{align}\label{sum_of_flux2}
    f_1+f_2=0.
\end{align}

Using (\ref{sum_of_flux}) or (\ref{sum_of_flux2}), we can obtain the following proposition (see also Figure \ref{fig:existence of a coordinate plane} for Proposition \ref{prop:existence of a coordinate plane}).
\begin{proposition}\label{prop:existence of a coordinate plane} The coordinate plane $\Pi$ that is perpendicular to $f_1$ and $f_2$ satisfies the following properties:
\begin{itemize}
\item $\Pi$ does not meet $\partial \Sigma$.
\item For a plane $\Pi'$ that is perpendicular to $\Pi$, it intersects $\partial\Sigma$ in at most four points.    
\end{itemize}
\end{proposition}

\begin{proof}
By Lemma \ref{lemm:byproduct of TPP}, $(\partial \Sigma)_1$ and $(\partial \Sigma)_2$ are strictly convex curves on $\mathbb{S}^2$.
Note that $\Pi$ is the coordinate plane in Proposition \ref{claim:flux} for $\mathcal{C}=(\partial \Sigma)_1, (\partial \Sigma)_2$, which implies the first statement. In addition, $\Pi'$ satisfies the property for a plane in Proposition \ref{prop:at_most_two} for $\mathcal{C}=(\partial \Sigma)_1, (\partial \Sigma)_2$, which implies the second statement. 
\end{proof}

\begin{figure}
\begin{minipage}[c]{0.55\linewidth}
\includegraphics[width=\linewidth]{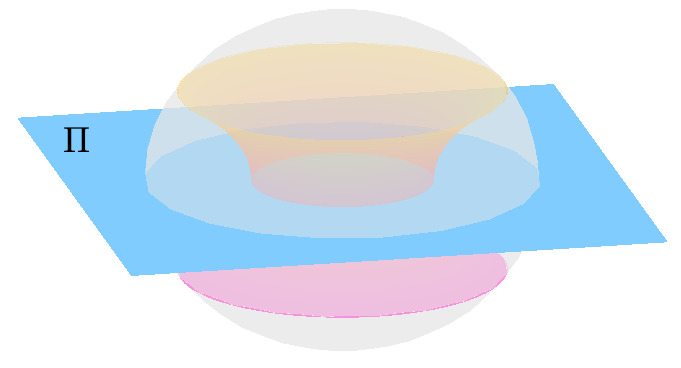}
\end{minipage}
\begin{minipage}[c]{0.4\linewidth}
\includegraphics[width=\linewidth]{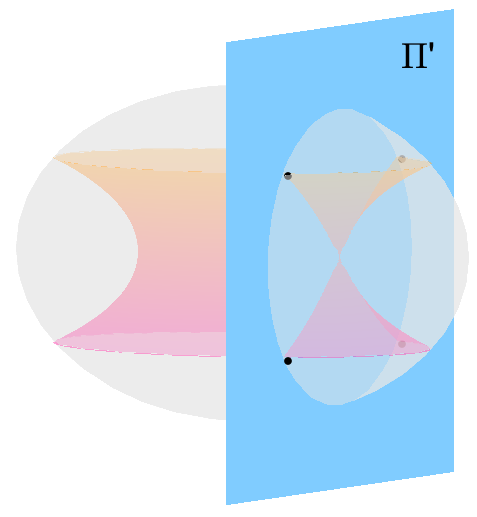}
\end{minipage}%
\caption{An illustration of Proposition \ref{prop:existence of a coordinate plane}. Let $\Pi$ be a coordinate plane perpendicular to $f_1$ and $\Pi'$ be a plane perpendicular to $\Pi$. Then, $\partial \Sigma$ does not intersect $\Pi$ (left picture) and intersects $\Pi'$ in at most four points (right picture).}
\label{fig:existence of a coordinate plane}
\end{figure}

\begin{proposition}\label{prop:plane passing through f1} If a coordinate plane $\Pi$ passes through $f_1$, $\Pi$ meets each of $(\partial \Sigma)_i, i=1,2$, at exactly two points.
\end{proposition}
\begin{proof}
By Proposition \ref{prop:interior} with Lemma \ref{lemm:byproduct of TPP}, $(\partial \Sigma)_1$ meets $\Pi$ at exactly two points. By (\ref{sum_of_flux}) or (\ref{sum_of_flux2}), the same is true for $(\partial \Sigma)_2$.
\end{proof}

\section{Symmetry conditions by planes}\label{sec:symmetry by planes}
In this section, we prove Theorem \ref{thm:interior symmetry} and Corollary \ref{cor:half-ball} by using propositions in the previous section. To this end, we will assume that $\sigma_1(\Sigma)<1$ and find a contradiction. Then, by the work of Fraser and Schoen, $\Sigma$ is congruent to the critical catenoid \cite{FS16SEB}.

We first recall some basic properties of a first Steklov eigenfunction, which are used in McGrath's work \cite{McG18ACC}.
\begin{definition}[Nodal domain]\label{def:nodal domain} Let $u$ be a Steklov eigenfunction of $\Sigma$. Then the nodal set of $u$ is $\mathcal{N}=\{\left.p\in \Sigma\right| u(p)=0\}$.
A nodal domain of $u$ is a component of $\Sigma\setminus\mathcal{N}$.
\end{definition}

\begin{lemma}[Courant nodal domain theorem, {\cite{KS69ISE}} or {\cite[Lemma 2.2]{McG18ACC}}] \label{lemm:Courant} If $u$ is a first Steklov eigenfunction, then $u$ has exactly two nodal domains.
\end{lemma}

\begin{lemma}[Symmetry of a first eigenfunction, {\cite[Lemma 3.2]{McG18ACC}}] \label{lemm:symmetry}If $\Sigma$ is invariant under the reflection through a plane and $\sigma_1(\Sigma)<1$, then a first Steklov eigenfunction is invariant under the reflection.
\end{lemma}

The following lemmas are obtained by Proposition \ref{prop:existence of a coordinate plane} and \ref{prop:plane passing through f1}.
\begin{lemma}\label{lemm:sign-changing}
If $\sigma_1(\Sigma)<1$, then a first Steklov eigenfunction $u$ is sign-changing in one of the boundary components of $\Sigma$. Furthermore, the nodal set of $u$ in this component of $\partial \Sigma$ is exactly two points.
\end{lemma}
\begin{proof}
Suppose $u$ does not change its sign in each boundary component of $\Sigma$. Since $\int_{\partial \Sigma}u=0$ by \ref{lemm:orthogonality} in Lemma \ref{lemm:Steklov eigenfunction}, $u$ has different signs in $(\partial \Sigma)_1$ and $(\partial \Sigma)_2$. By Proposition \ref{prop:existence of a coordinate plane}, we can find a plane $\{v\cdot x=0\}$ that does not meet $\partial \Sigma$. Since $v \cdot x$ has different signs in $(\partial \Sigma)_1$ and $(\partial \Sigma)_2$, we have
\begin{align} \label{orthogonal to a plane1}
    \int_{\partial \Sigma} (v\cdot x)u  \neq 0.
\end{align}
On the other hand, by \ref{lemm:orthogonality} in Lemma \ref{lemm:Steklov eigenfunction}, $\sigma_1(\Sigma)<1$ implies that $\int_{\partial \Sigma}(v\cdot x)u=0$, which leads a contradiction with (\ref{orthogonal to a plane1}). 
Therefore, $u$ is sign-changing in at least one component of $\partial \Sigma$, say $(\partial \Sigma)_1$. Then, $(\partial \Sigma)_1\cap \mathcal{N}$ has at least two points. By Remark \ref{rmk:at most two by first eigenfunction}, $(\partial \Sigma)_1$ has exactly two points in $\mathcal{N}$ and this finishes the proof.
\end{proof}

\begin{lemma} \label{lemm:reflection planes}If $\partial \Sigma$ is invariant under the reflection through a plane $\Pi$, then $\Pi$ is either a coordinate plane passing through $f_1$ or a coordinate plane perpendicular to $f_1$. In particular, the former case is equivalent to a reflection plane that meets $\partial \Sigma$ and the latter case is equivalent to a reflection plane that does not meet $\partial \Sigma$.
\end{lemma}
\begin{proof}
Let $\Pi$ be a symmetry plane of $\Sigma$ and let $R_{\Pi}$ be the reflection through $\Pi$. Then, $R_{\Pi}$ satisfies one of the following:
\begin{multicols}{2}
\begin{itemize}
    \item $R_{\Pi}(f_1)=f_1$, $R_{\Pi}(f_2)=f_2$,
    
    \item $R_{\Pi}(f_1)=f_2$. 
\end{itemize}
\end{multicols}
In both cases, $\Pi$ should be a coordinate plane. Indeed, for the former case, $\Pi$ is a coordinate plane passing through $f_1$. By Proposition \ref{prop:plane passing through f1}, $\Pi$ meets $\partial \Sigma$. For the latter case, $\Pi$ is the coordinate plane perpendicular to $f_1$. In this case, $\Pi$ does not meet $\partial \Sigma$ by Proposition \ref{prop:existence of a coordinate plane}.
\end{proof}

Using previous lemmas with Proposition \ref{prop:existence of a coordinate plane} and \ref{prop:plane passing through f1}, we can obtain the following theorem.\\

\textit{Proof of Theorem \ref{thm:interior symmetry}}

Suppose $\sigma_1(\Sigma)<1$. Let $u$ be a first Steklov eigenfunction and $\mathcal{N}$ be the nodal set of $u$. Using Lemma \ref{lemm:sign-changing}, we may assume that $u$ is sign-changing in $(\partial \Sigma)_1$ and let $p_1, p_2 \in \mathcal{N}\cap (\partial \Sigma)_1$.

For simplicity, we may assume $f_1=(0,0,1)$. By Lemma \ref{lemm:reflection planes}, it is enough to consider the following two cases.

\textbf{Case 1} (The Reflection planes are coordinate planes passing through $f_1$). \\
By Proposition \ref{prop:plane passing through f1}, each of the reflection planes meets $(\partial \Sigma)_1$ at exactly two points. If the reflection planes are not orthogonal to each other, by symmetry of $u$ (Lemma \ref{lemm:symmetry}), there are at least three nodal points in $(\partial \Sigma)_1$, which contradicts Lemma \ref{lemm:sign-changing}. Now we assume two reflection planes are orthogonal to each other and let $\Pi_1=\{x_1=0\}$ and $\Pi_2=\{x_2=0\}$. Then, $\mathcal{N}\cap (\partial\Sigma)_1\subset \Pi_1 \cup \Pi_2$, because if it is not true, $(\partial \Sigma)_1$ contains at least four nodal points by the symmetry of $u$, contrary to Lemma \ref{lemm:sign-changing}. Now, assume $p_1\in\Pi_1$. Then, $p_2\in \Pi_1$ by the symmetry of $u$. Then, also by the symmetry of $u$, $u$ does not change its sign in $(\partial \Sigma)_1$, contrary to our assumption. 

\textbf{Case 2} (A reflection plane is $\{x_3=0\}$).\\
Let $\Pi_1'=\{x_3=0\}$. By the symmetry of $u$ (see Lemma \ref{lemm:symmetry}), we have $p_3, p_4 \in \mathcal{N}\cap (\partial \Sigma)_2$ and $R_{\Pi_1'}(\{p_1, p_2\})=\{p_3, p_4\}$. Then, we have a plane $\Pi_2'$ passing through $p_1, p_2, p_3, p_4$ which is perpendicular to $\Pi_1'$. Let $\Pi_2'= \{x_2=c\}$. By Proposition \ref{prop:existence of a coordinate plane}, $|\Pi_2' \cap \partial\Sigma|\le 4$, thus $\Pi_2'\cap \partial \Sigma=\{p_1, p_2, p_3, p_4\}$. By the symmetry of $u$, the sign of $u$ does not change at each $\left((\partial \Sigma)_1\cap\{x_2>c\}\right)\cup\left((\partial \Sigma)_2\cap\{x_2>c\}\right)$ and $\left((\partial \Sigma)_1\cap\{x_2<c\}\right)\cup\left((\partial \Sigma)_2\cap\{x_2<c\}\right)$. Note that the two signs are different because of the definition of $p_1$ and $p_2$. Thus, we have
\begin{align}\label{orthogonal to a plane2}
    \int_{\partial \Sigma}u(x_2-c)  \neq 0.
\end{align}
On the other hand, by \ref{lemm:orthogonality} in Lemma \ref{lemm:Steklov eigenfunction}, $\sigma_1(\Sigma)<1$ implies that $\int_{\partial \Sigma}u=\int_{\partial \Sigma}ux_2=0$, which leads a contradiction with (\ref{orthogonal to a plane2}). 

Thus, in both of Case 1 and 2, we have $\sigma_1(\Sigma)=1$. Therefore, by the theorem of Fraser and Schoen (see \cite[Theorem 1.2]{FS16SEB}), $\Sigma$ is congruent to the critical catenoid.
\qed\\

\textit{Proof of Corollary \ref{cor:half-ball}}

$\Sigma'$ can be extended to the embedded free boundary minimal annulus in the unit ball in $\mathbb{R}^3$ with the reflection symmetry through the plane. By the second statement in Theorem \ref{thm:interior symmetry}, the surface is the critical catenoid, which is the desired conclusion. \qed  

\begin{remark}\label{rmk:boundary condition}
The boundary condition of Corollary \ref{cor:half-ball}, each boundary component of $\Sigma'$ is contained in either an open hemisphere or equatorial disk, would be necessary. In \cite{CFS22CJM}, Carlotto, Franz, and Schulz constructed a compact embedded free boundary minimal surface $\Sigma_{1,1}$ of genus 1 in $\mathbb{B}^3$ with only one boundary component. Note that $\Sigma_{1,1}$ has the dihedral symmetry $D_{2}$. The author expects that there is a coordinate plane $\Pi$ such that $\Sigma_{1,1}$ would be invariant under the reflection through $\Pi$ and a component of $\Sigma_{1,1}\setminus\Pi$ would be an embedded free boundary minimal annulus in one of the half-balls bounded by $\Pi$ (see Figure \ref{fig:half ball}). 
\end{remark}

\begin{figure}
    \centering
    \includegraphics[width=0.8\textwidth]{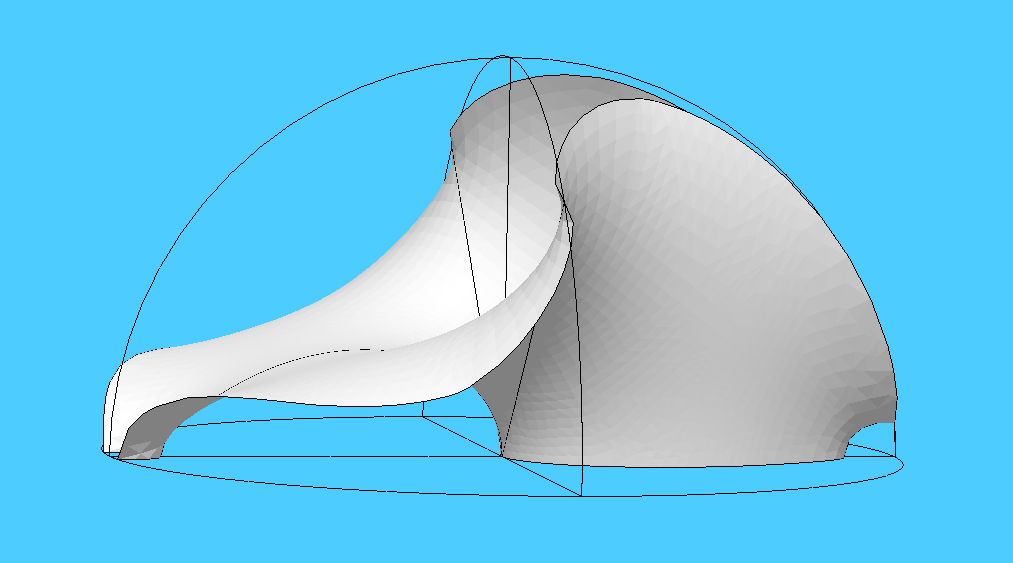}
    \caption{An expected image of a free boundary minimal annulus in a half-ball. Note that one boundary component intersects both the open hemisphere and the equatorial disk.}
    \label{fig:half ball}
\end{figure}

\section{Symmetry properties}\label{sec:symmetry principle}
In this section, we will consider symmetric properties of a compact embedded free boundary minimal surface in $\mathbb{B}^3$ (Theorem \ref{thm:symmetry principle}, Lemma \ref{lemm:boundary symmetry}). The first theorem roughly says that a symmetry of the boundary of a compact embedded free boundary minimal surface in $\mathbb{B}^3$ engenders a global symmetry. It is reminiscent of the symmetry principle for Bryant surfaces (see \cite[Theorem 12]{GM05CLB}).

We will prove Theorem \ref{thm:symmetry principle} (symmetry principle) by the Bj\"{o}rling uniqueness theorem (see, for example, \cite[Theorem 1 in p.125]{DHS10MS}). Before we prove the theorem, we observe the Gauss map of $M$ along $\partial M$. Note that a compact embedded free boundary minimal surface in $\mathbb{B}^3$ is orientable, hence the Gauss map of $M$ is well-defined (see \cite[Proposition 2.6]{Li20FBM}).

\begin{lemma}\label{lemm:pre-boundary symmetry} Let $(\partial M)_1, \dots, (\partial M)_n$ be the boundary components of $M$, where $M$ is defined in Theorem \ref{thm:symmetry principle}. If $A((\partial M)_i)=(\partial M)_j$ for some $1\le i,j\le n$ ($i$ can be identical to $j$), then the Gauss map $\nu$ of $M$ satisfies
\begin{align} \label{pre-boundary symmetry}
    \nu\circ A(p) = \pm A\circ \nu(p)
\end{align}
        for all $p\in (\partial \Sigma)_i$. Note that the signs are identical for all $p\in (\partial \Sigma)_i$.
\end{lemma}
\begin{proof}
Note that $\nu(p)$ is perpendicular to a tangent vector of $(\partial M)_i$ at $p$, say $t(p)$, and to the outward conormal vector of $M$ along $(\partial M)_i$ at $p$, say $\eta(p)=p$. 
Then, 
\begin{align}
    A\circ \nu(p) \perp A\circ t(p), A\circ \eta(p).
\end{align}
Since $A$ is linear, $A\circ t(p)$ is a tangent vector of $(\partial M)_j$ at $A(p)$. Furthermore, $A\circ \eta(p)=A(p)$ is the outward unit conormal vector of $M$ along $(\partial M)_j$ at $A(p)$. Thus, $\nu\circ A(p)$ is parallel to $A\circ \nu(p)$. Since the function
\begin{align}
    \left<\nu\circ A(\cdot), A\circ \nu(\cdot)  \right> : (\partial M)_i \rightarrow \{1,-1\}
\end{align}
is continuous, the sign in (\ref{pre-boundary symmetry}) is same for all $p\in (\partial \Sigma)_i$. 
\end{proof}

\textit{Proof of Theorem \ref{thm:symmetry principle}}

Note that $\partial M$ is real analytic (see \cite[section 6]{Lew1951MSP}).  By Lemma \ref{lemm:pre-boundary symmetry}, $A$ maps a Bj\"{o}rling data of $(\partial M)_i$ into those of $(\partial M)_j$ for some $i,j$. Then, by the Bj\"{o}rling uniqueness theorem, $A$ maps a Bj\"{o}rling solution around $(\partial M)_i$ into a Bj\"{o}rling solution around $(\partial M)_j$. We can apply this argument multiple times so that $M$ is covered by a union of Bj\"{o}rling solutions. Now we need to check whether our union of Bj\"{o}rling solutions in $M$ is $A$-invariant. Let $U_1$ and $U_2$ be Bj\"{o}rling solutions such that $A(U_1)=U_2$. Since $M$ can be continued analytically across $\mathbb{S}^2$ (see \cite[section 6]{Lew1951MSP}), we may assume $U_i\cap M=U_i \cap \bar{\mathbb{B}}^3$ for all $i=1,2$. Then,
\begin{align}
    A(U_1\cap M)=A(U_1\cap \bar{\mathbb{B}}^3)=U_2\cap \bar{\mathbb{B}}^3=U_2\cap M.
\end{align}
Therefore, we can conclude $A(M)=M$.
\qed

\begin{remark}
By Theorem \ref{thm:symmetry principle}, if a boundary component of $M$ is invariant under the reflection through a coordinate plane, then $M$ is invariant under the reflection. It implies that if a component of $\partial M$ is rotationally symmetric, $M$ is rotationally symmetric. Then, one can show that $M$ is the equatorial disk or the critical catenoid (see \cite[Corollary 3.9]{KL21FBM}).
\end{remark}

By the symmetry principle and Theorem \ref{thm:interior symmetry}, the following corollary is straightforward.

\begin{corollary} \label{cor:boundary condition by planes} If $\partial\Sigma$ has one of following symmetry conditions, then $\Sigma$ is congruent to the critical catenoid.
\begin{itemize} 
    \item $\partial\Sigma$ is invariant under the reflections through two distinct planes. 
    \item $\partial\Sigma$ is invariant under the reflection through a plane that does not meet $\partial \Sigma$.
\end{itemize}
\end{corollary} 

We will obtain a stronger lemma (Lemma \ref{lemm:boundary symmetry}) than Lemma \ref{lemm:pre-boundary symmetry} for $M_0$, which will be used in Section \ref{sec:congruent}. The strategy of the lemma is analyzing the direction of the Gauss map along $(\partial M_0)_i$ (Proposition \ref{prop:good Gauss map}). The author thanks José A. Gálvez for providing Lemma \ref{lemm: cone condition} that makes the proof of Proposition \ref{prop:good Gauss map} simpler than the previous proof.

\begin{lemma}\label{lemm: cone condition}
Let $D_i, i=1,\dots, b$, be the strictly convex spherical sets whose boundaries are $(\partial M_0)_i, i=1,\dots, b$, respectively (see Proposition \ref{prop:interior}). Then, for all $i$, there is a small neighborhood of $(\partial M_0)_i$ in $M_0$ that is contained in outside of the convex cone over $D_i$.
\end{lemma}

\begin{proof}
Let $\nu$ be an oriented Gauss map of $M_0$.
Take a point $p\in (\partial M_0)_i$ and consider the normal plane $\Pi$ at $p$ in direction $x(p)$, which is the plane spanned by $x(p)$ and $\nu(p)$. Since the mean curvature of $M_0$ (or analytic continuation of $M_0$ across $\mathbb{S}^2$, see \cite[section 6]{Lew1951MSP}) at $p$ is zero, the normal curvature in direction $x(p)$ has different sign with the normal curvature in direction of the tangent vector of $(\partial M_0)_i$ at $p$. In addition, since the normal curvature in direction of the tangent vector of $(\partial M_0)_i$ at $p$, which is equal to the curvature of $(\partial M_0)_i$ on $\mathbb{S}^2$ at $p$ up to sign, is nonzero by the strictly convexity of $(\partial M_0)_i$, so does the normal curvature in direction $x(p)$. Thus, a small neighborhood of $\Pi\cap M_0$ around $p$ is contained in outside of the convex cone over $D_i$, which implies the conclusion.
\end{proof}

\begin{proposition} \label{prop:good Gauss map} With the same definition of $D_i, i=1,\dots, b,$ in Lemma \ref{lemm: cone condition}, there is a Gauss map $\nu$ of $M_0$ such that $\left.\nu\right|_{(\partial M_0)_i}$ is the outward unit conormal vector of $D_i$ along $\partial D_i$ for all $i$.
\end{proposition}
\begin{proof}

First, as in the proof of Lemma \ref{lemm: cone condition}, we consider the normal plane $\Pi$ in direction $x(p)$ at $p\in (\partial M_0)_1$. Then, $\Pi \cap M_0$ contains a curve $\gamma(s), s\in [0,s_1]$ with $\gamma(0)=p$ located outside of the convex cone over $D_1$. For convenience, we consider a $xyz$-coordinate system of $\mathbb{R}^3$ so that $\Pi$ becomes the $xy$-plane. We may assume $\gamma(0)=(0, 1, 0)$ and the intersection of the convex cone and $\Pi$ lies in $\{x\ge 0\}$. Let $t(s), n(s)$ be the tangent vector and the normal vector of $\gamma$ at $\gamma(s)$, respectively. Then, $t(0)=-e_2$ and $n(0)=e_1$. See Figure \ref{fig: outer curve} for this setting. Then, for sufficiently small $s_1$,
\begin{align} \label{normal vector}
    \left<n(s), e_1\right> >0, \left<n(s), e_2\right><0, \text{ and } \left<n(s), e_3\right>=0  
\end{align}
for all $s\in (0, s_1]$.


\begin{figure}
    \centering
    \begin{tikzpicture}
   \draw[fill=gray!10] (-6,-0.5)--(6,-0.5)--(6,5)--(-6,5) --cycle;
   \draw [domain=30:140] plot ({4*cos(\x)}, {4*sin(\x)});
   \draw [ultra thick,domain=40:90] plot ({4*cos(\x)}, {4*sin(\x)});
   \draw[fill=green!20] (0,0) --  (40:4) arc(40:90:4) -- cycle;

   \draw[domain=-80:0] plot ({-1.2+1.2*cos(\x)}, {4+1.8*sin(\x)});
   \node[left] at ({-1.2+1.2*cos(40)}, {4-1.8*sin(40)})  {$\gamma$}; 
   \draw (-.2,4)|-(0,3.8);
   
   \draw[->,thick] ({-1.2+1.2*cos(80)}, {4-1.8*sin(80)})--({-1.2+1.2*cos(80)-0.6*sin(80)}, {4-1.8*sin(80)-0.9*cos(80)}) node[left]{$t(s)$};
   \draw[->,thick] ({-1.2+1.2*cos(80)}, {4-1.8*sin(80)})--({-1.2+1.2*cos(80)+0.9*cos(80)}, {4-1.8*sin(80)-0.6*sin(80)}) node[below]{$n(s)$};
   \draw ({-1.2+1.2*cos(80)-0.2*sin(80)}, {4-1.8*sin(80)-0.3*cos(80)}) -- ({-1.2+1.2*cos(80)-0.2*sin(80)+0.3*cos(80)}, {4-1.8*sin(80)-0.3*cos(80)-0.2*sin(80)}) --({-1.2+1.2*cos(80)+0.3*cos(80)}, {4-1.8*sin(80)-0.2*sin(80)});
   
   \fill [black] (0,4) circle (1.5pt) node[above]{$(0,1,0)$};
   \node[above, xshift=0.3cm] at ({4*cos(65)}, {4*sin(65)}) {$D_1\cap \Pi$};
   \fill [black] (0,0) circle (1.5pt) node[left]{$o$};
   \node[right] at ({4*cos(30)}, {4*sin(30)}) {$\Pi\cap \partial \mathbb{B}^3$};
   \node at (-5.5, 4.5) {$\Pi$};
   
   \draw[->] (4.5,3.5)--(4.5, 4.5) node[left]{$y$};
   \draw[->] (4.5, 3.5) -- (5.5, 3.5) node[below]{$x$};

\end{tikzpicture}

    \caption{Setting in the proof of Proposition \ref{prop:good Gauss map}. We take a normal plane $\Pi$ at $p\in (\partial M_0)_1$ in direction $x(p)$. Then, by Lemma \ref{lemm: cone condition}, $\Pi\cap M_0$ contains a curve $\gamma(s), s\in [0,s_1]$ with $\gamma(0)=p$ that is not contained in the convex cone over $D_1$. We consider a $xyz$-coordinate system of $\mathbb{R}^3$ so that $\Pi$ becomes the $xy$-plane and assume $\gamma(0)=(0,1,0)$. In addition, we may assume the circular sector that is the intersection of the convex cone over $D_1$ with $\Pi$ lies in $\{x\ge 0\}$ (see the green region in the figure). The thick circular arc in the figure is the intersection of $D_1$ and $\Pi$. Let $t(s)$ and $n(s)$ be the tangent vector and the normal vector of $\gamma$ at $\gamma(s)$, respectively. Note that $t(0)=-e_2$ and $n(0)=e_1$.}
    \label{fig: outer curve}
\end{figure}
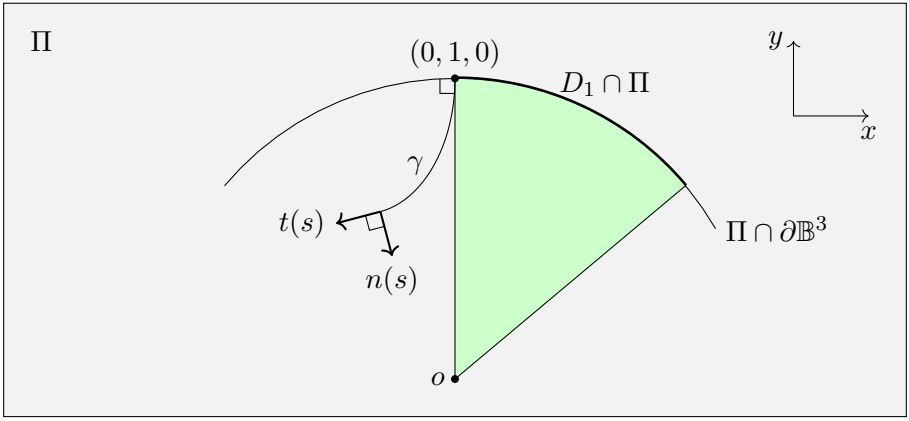



Let $\nu$ be an oriented Gauss map of $M_0$. By Corollary \ref{cor: transversality}, any coordinate planes cannot meet the interior of $M_0$ tangentially. Thus, $o\notin M_0$, where $o$ is the origin, and $\left<x, \nu \right>\neq 0$ in the interior of $M_0$, where $x$ is the position vector in $\mathbb{R}^3$. Now, we may assume 
\begin{align}\label{direction of the gauss map}
    \left<x, \nu \right>>0 \quad \text{ in }\text{Int}(M_0).
\end{align}

We prove $\nu$ satisfies the statement by proof by contradiction. Assume that $\nu$ is coincide with the inward unit conormal vector $\mathbf{n}_{in}$ of $D_1$ along $\partial D_1=(\partial M_0)_1$ at $p$. Then, we have $\nu(0)=n(0)$, so $\left<\nu(0), e_1\right>>0$ and $\left|\left<\nu(0), e_3\right>\right| \neq 1$. Here, we denote $\nu(\gamma(s))$ by $\nu(s)$ for simplicity. By continuity, we have 
\begin{align} \label{gauss vector}
     \left<\nu(s), e_1\right>>0 \text{ and } \left|\left<\nu(s), e_3\right>\right|\neq 1 \quad \text{ for all } s\in [0, s_2].
\end{align}
for sufficiently small $s_2<s_1$.
$\nu(s)$ lies in the unit circle $C(s), s\in[0, s_2]$ in the plane perpendicular to $t(s)$. Note that $C(s)$ contains $e_3$ and $-e_3$. Then, $C(s)-\{e_3,-e_3\}$ consists of the two components, say $C_1(s)$ and $C_2(s)$. We may assume $n(s)\in C_1(s)$ for all $s\in[0,s_2]$ by continuity. Clearly, $n(s)$ is the midpoint of $C_1(s)$. Since $\nu(0)=n(0)=e_1$, $\nu(0)\in C_1(0)$. Furthermore, by the inequation in (\ref{gauss vector}), $\nu(s)\in C_1(s)$ for all $s\in [0,s_2]$. Since each point in $C_1(s), s\in (0,s_2]$ has negative $y$-coordinate by (\ref{normal vector}), we have
\begin{align} \label{gauss vector 2}
    \left<\nu(s), e_2\right><0 \quad \text{ for all } s\in (0, s_2].
\end{align}
On the other hand, the position vector $x$ satisfies
\begin{align}\label{position vector}
    \left<x(s), e_1\right><0, \left<x(s), e_2\right>>0, \text{ and } \left<x(s), e_3\right>=0,
\end{align}
where $x(s)=x(\gamma(s))$ for all $s\in (0,s_2]$. Thus, by (\ref{gauss vector}), (\ref{gauss vector 2}), and (\ref{position vector}), we have 
\begin{align}
    \left<x(s),\nu(s)\right>=\sum_{i=1}^3\left<x(s), e_i\right>\cdot \left<\nu(s), e_i\right> <0 \quad \text{ for all } s\in (0, s_2], 
\end{align}
which contradicts (\ref{direction of the gauss map}).


\end{proof}

\begin{lemma}[Boundary symmetry]\label{lemm:boundary symmetry}If $A((\partial M_0)_i)=(\partial M_0)_j$  for some $1\le i,j
\le b$ ($i$ can be identical to $j$), then the Gauss map $\nu$ of $M_0$ satisfies
\begin{align} \label{boundary symmetry}
    \nu\circ A(p) = A\circ \nu(p)
\end{align}
        for all $p\in (\partial M_0)_i$.
\end{lemma}
\begin{proof}
We may assume $\nu$ satisfies the property in Proposition \ref{prop:good Gauss map}. Since $A$ is an isometry of $\mathbb{S}^2$, we have $A(D_i)=D_j$. Then, $A$ maps the outward unit conormal vector of $D_i$ along $\partial D_i$ at $p$ into the outward unit conormal vector of $D_j$ along $\partial D_j$ at $A(p)$. Thus, we obtain (\ref{boundary symmetry}).  
\end{proof}

\section{Conditions with congruent boundary components}\label{sec:congruent}
We give a proof of Theorem \ref{thm:congruent with one symmetry} by combining several results in this paper. The proof is based on the observation about the congruence of the piece of $\Sigma$ in some hemisphere and the piece of $\Sigma$ in the other hemisphere.\\

\textit{Proof of Theorem \ref{thm:congruent with one symmetry}}

We may assume that $f_1=(0,0,1)$ and $f_2=(0,0,-1)$ by (\ref{sum_of_flux2}). Since $(\partial \Sigma)_1$ and $(\partial \Sigma)_2$ are congruent, there is $A\in O(3)$ such that $A((\partial \Sigma)_1)=(\partial \Sigma)_2$. Let $\Sigma_1= \Sigma\cap \{x_3\ge 0\}$ and $\Sigma_2=\Sigma\cap \{x_3\le0\}$. By the symmetry principle (Theorem \ref{thm:symmetry principle}), $A(\Sigma_1)=\Sigma_2$. Furthermore, $A\circ \nu$ is the Gauss map of $\Sigma$ and it has the same direction with $\nu\circ A$ by the boundary symmetry (Lemma \ref{lemm:boundary symmetry}). Thus, for any $p\in \Sigma \cap \{x_3=0\}$, we have
\begin{align}\label{applying A}
    A(p)\in \Sigma\cap \{x_3=0\} \quad \text{and} \quad \nu\circ A(p) = A\circ \nu(p).
\end{align}

Since $A(f_1)=f_2$, $A$ can be represented by $B\circ R_{\{x_3=0\}}$, where $B$ is a map in $O(3)$ that preserves $x_3$-axis and $R_{\{x_3=0\}}$ is the reflection through $\{x_3=0\}$. Note that $B$ is either a rotation about $x_3$-axis or the reflection through a plane that passes through $f_2$ and the origin.

\textbf{Case 1} ($B$ is a rotation about $x_3$-axis).\\
Let $B$ be the rotation about $x_3$-axis by an angle $\theta$. Let $p\in \Sigma \cap \{x_3=0\}$ and $\nu(p)=(x,y,z)$. Refer to Figure \ref{fig:congruent with one symmetry} for a basic strategy for a specific case.

If $B$ is an irrational rotation, the orbit of $p$ under $A$ is dense in the circle in $\{x_3=0\}$ of radius $|p|$, centered at the origin. By (\ref{applying A}), the orbit of $p$ is contained in $\Sigma\cap \{x_3=0\}$, which implies the circle is identical with $\Sigma \cap \{x_3=0\}$ by the transversality of $\Sigma$ with $\{x_3=0\}$ (Corollary \ref{cor: transversality}) and the two-piece property (Theorem \ref{thm:two-piece property}). In addition, $x_3$-coordinates of its Gauss map image are either $z$ or $-z$. Note that $\{\left.q\in \{A^n(p), n\in \mathbb{N}\}\right|\nu(q)\cdot (0,0,1)=z\}$ and $\{\left.r\in \{A^n(p), n\in \mathbb{N}\}\right|\nu(r)\cdot (0,0,1)=-z\}$ are dense in the circle. By the continuity of $\nu$, $z=0$ and the Gauss map image of the circle lies in $\{x_3=0\}$. By Corollary \ref{cor:half-ball}, $\Sigma$ is congruent to the critical catenoid.   

Now, we may assume $\theta=\frac{b}{a}\cdot 2\pi$, where $a$ and $b$ are relatively prime.

If $a$ is odd, applying (\ref{applying A}), we obtain $\nu(p)=\nu\circ A^a(p)=A^a\circ \nu(p)=(x,y,-z)$. Then, $z=0$ (see Figure \ref{fig:congruent with one symmetry}). Therefore, by Corollary \ref{cor:half-ball}, $\Sigma$ is congruent to the critical catenoid. 

If $a=2(2k-1), k\in \mathbb{N}$, applying (\ref{applying A}), we obtain $-p=A^{a/2}(p)\in \Sigma\cap \{x_3=0\}$ and $\nu(-p)=\nu\circ A^{a/2}(p)=A^{a/2}\circ \nu(p)=-\nu(p)$. By the Bj\"{o}rling uniqueness theorem (see \cite[Theorem 1 in p.125]{DHS10MS} or the proof of Theorem \ref{thm:symmetry principle}), $\Sigma$ is invariant under the antipodal map. Then, by Corollary \ref{cor: antipodal symmetry on the boundary}, $\Sigma$ is congruent to the critical catenoid. 

If $a=2^m(2k-1), m, k\in \mathbb{N}$ with $m\ge 2$, we consider the symmetry assumption on $\partial \Sigma$. By Lemma \ref{lemm:reflection planes} and Corollary \ref{cor:boundary condition by planes}, we may assume $\Pi_1=\{x_1=0\}$. Then, by the symmetry principle, $\Sigma$ is invariant under $R_{\{x_1=0\}}$. Now, by (\ref{applying A}), we obtain $\nu(-p)=\nu\circ A^{a/2}(p)=A^{a/2}\circ \nu(p)=(-x,-y,z)$. Then we have $R_{\{x_2=0\}}(p)=R_{\{x_1=0\}}(-p)\in \Sigma \cap \{x_3=0\}$ and $\nu\circ R_{\{x_2=0\}}(p)=\nu\circ R_{\{x_1=0\}}(-p)=R_{\{x_1=0\}}\circ \nu(-p)=R_{\{x_1=0\}}(-x,-y,z)= R_{\{x_2=0\}}\circ \nu(p)$. Then, by the Bj\"{o}rling uniqueness theorem, $\Sigma$ is also invariant under $R_{\{x_2=0\}}$. Thus, by Theorem \ref{thm:interior symmetry}, $\Sigma$ is congruent to the critical catenoid.   

\begin{figure}
\subfigure[]{
    \begin{tikzpicture}
\draw (0,0) ellipse (2 and 1);
\node at (2.5, -1) {$\Sigma \cap \{x_3=0\}$};

\fill [black] ($(0,0)+(30:2 and 1)$) circle (1.5pt) node[anchor=north east] {$p$};
\fill [black] ($(0,0)+(150:2 and 1)$) circle (1.5pt) node[anchor=north west]{$A(p)$};
\fill [black] ($(0,0)+(270:2 and 1)$) circle (1.5pt) node[anchor=north]{$A^{2}(p)$};

\draw[->, blue, thick] ($(0,0)+(30:2 and 1)$) --($(0,0.5)+(30:2 and 1)$);

\node[anchor=west, blue] at ($(0,0.25)+(30:2 and 1)$){$\nu(p)$};
\end{tikzpicture}
}
\subfigure[]{
\begin{tikzpicture}
\draw (0,0) ellipse (2 and 1);
\node at (2.5, -1) {$\Sigma \cap \{x_3=0\}$};

\fill [black] ($(0,0)+(30:2 and 1)$) circle (1.5pt) ;
\fill [black] ($(0,0)+(150:2 and 1)$) circle (1.5pt) node[anchor=north west, blue]{$\nu\circ A(p)$};
\fill [black] ($(0,0)+(270:2 and 1)$) circle (1.5pt) node[anchor=north, blue]{$\nu\circ A^{2}(p)$};

\draw[->, blue, thick] ($(0,0)+(30:2 and 1)$) --($(0,0.5)+(30:2 and 1)$);
\node[anchor=west, blue] at ($(0,0.25)+(30:2 and 1)$){$\nu(p)$};
\draw[->, blue, thick] ($(0,0)+(150:2 and 1)$) --($(0,-0.5)+(150:2 and 1)$);
\draw[->, blue, thick] ($(0,0)+(270:2 and 1)$) --($(0,0.5)+(270:2 and 1)$);
\draw[->, blue, thick, dotted] ($(0,0)+(30:2 and 1)$) --($(0,-0.5)+(30:2 and 1)$);
\end{tikzpicture}
}
    \caption{Illustrations describing an argument for $B=\mathbb{Z}_3$ about $x_3$-axis in Case 1 in the proof of Theorem \ref{thm:congruent with one symmetry}. We show that $(\nu(p))_3=0$ at $p\in \Sigma\cap \{x_3=0\}$, where $(v)_3$ is the $x_3$-coordinate of $v\in \mathbb{R}^3$. Suppose not and assume $(\nu(p))_3>0$. By (\ref{applying A}), $A=B\circ R_{\{x_3=0\}}$ satisfies $A(p),A^2(p)\in \Sigma\cap \{x_3=0\}$ (see the first picture). In addition, we have $((\nu\circ A)(p))_3<0, ((\nu\circ A^2)(p))_3>0$, and $((\nu\circ A^3)(p))_3=\nu(p)<0$, which contradicts our assumption (see the second picture). It implies that $\Sigma$ intersects $\{x_3=0\}$ perpendicularly. Therefore, by Corollary \ref{cor:half-ball}, $\Sigma$ is congruent to the critical catenoid.}
    \label{fig:congruent with one symmetry}
\end{figure}
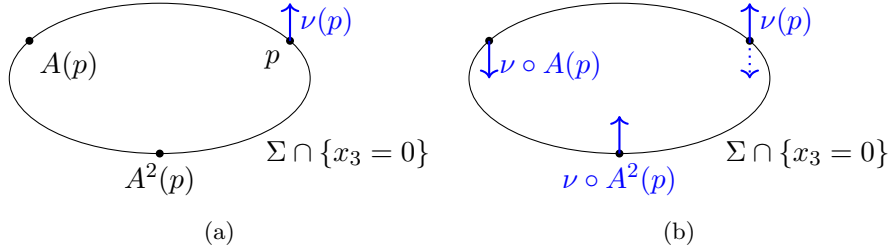

\textbf{Case 2} ($B$ is the reflection through a plane passing through $f_2$ and the origin).\\
Let $\Pi_1$ be the reflection plane of $B$ and $\Pi_2$ be the reflection plane in the assumption on $\partial \Sigma$. Using Lemma \ref{lemm:reflection planes} and Corollary \ref{cor:boundary condition by planes}, we may assume $\Pi_2:=\{x_1=0\}$. Then, $R_{\Pi_1}\circ R_{\Pi_2}$ is a rotation about $x_3$-axis. Then, we define $B':=R_{\Pi_1}\circ R_{\Pi_2}=B\circ R_{\Pi_2}$ and $A':=B'\circ R_{\{x_3=0\}}=A\circ R_{\Pi_2}$. Then (\ref{applying A}) holds for $A'$ instead of $A$, because the symmetry principle and the boundary symmetry (see Lemma \ref{lemm:boundary symmetry}) give $A'(p)=B'\circ R_{\{x_3=0\}}(p)=B\circ R_{\Pi_2}(p)=A\circ R_{\Pi_2}(p)\in \Sigma\cap \{x_3=0\}$ and $\nu\circ A'(p)=\nu\circ A\circ R_{\Pi_2}(p)=(A\circ \nu)\circ R_{\Pi_2}(p)=A\circ (R_{\Pi_2}\circ \nu)(p)=B\circ (R_{\Pi_2}\circ R_{\{x_3=0\}})\circ \nu(p)=A'\circ \nu(p)$. Thus, Case 2 reduces to Case 1.  

Therefore, by case study, we showed that $\Sigma$ is congruent to the critical catenoid. \qed

The following corollary is straightforward from Case 1 in the proof of Theorem \ref{thm:congruent with one symmetry}.

\begin{corollary} \label{cor: rotoreflection}
Let $A:=B\circ R\in O(3)$ be a rotoreflection symmetry, where $R$ is the reflection through a plane $\Pi$ with $R(f_1)=f_2$, and $B$ is a rotation about the axis perpendicular to $\Pi$. Assume $B$ is either
\begin{enumerate}
    \item an irrational rotation, or
    \item a rotation by an angle $\theta:=\frac{b}{a}\cdot \pi$, where $a$ is an odd number and $a$ and $b$ are relatively prime.
\end{enumerate}
If $\partial \Sigma$ is invariant under $A$, then $\Sigma$ is congruent to the critical catenoid.
\end{corollary}

Note that $A$ becomes the antipodal map if $B$ is the rotation by an angle $\pi$, so Corollary \ref{cor: rotoreflection} is a generalization of Corollary \ref{cor: antipodal symmetry on the boundary}.

\section{Uniqueness of the critical catenoid}\label{sec: uniqueness of the critical catenoid}
In this section, we summarize our sufficient conditions for $\Sigma$ to be the critical catenoid. 

\textbf{A Condition on a component of $\partial \Sigma$}
\begin{outline}
    \1 The reflection symmetries through two distinct planes (Corollary \ref{cor:two plane with one component}).
\end{outline}

\textbf{Conditions on $\partial \Sigma$}
\begin{outline}
    \1 The reflection symmetries through two distinct planes (Corollary \ref{cor:boundary condition by planes}).
    \1 The reflection symmetry through a plane with additional conditions.
        \2 The reflection plane $\Pi$ does not meet $\partial \Sigma$ (Corollary \ref{cor:boundary condition by planes}).
        \2 The reflection plane $\Pi$ intersects $\partial \Sigma$ and the two components of $\partial \Sigma$ are congruent (Theorem \ref{thm:congruent with one symmetry}).
    \1  The rotoreflection symmetry by $A:=B\circ R\in O(3)$, where $R$ is the reflection through a plane $\Pi$ with $R(f_1)=f_2$, and $B$ is a rotation about the axis perpendicular to $\Pi$.
        \2 $B$ is an irrational rotation (Corollary \ref{cor: rotoreflection}).
        \2 $B$ is a rotation by an angle $\theta:=\frac{b}{a}\cdot \pi$, where $a$ is an odd number and $a$ and $b$ are relatively prime (Corollary \ref{cor: rotoreflection}). 
\end{outline}

By the symmetry principle (Theorem \ref{thm:symmetry principle}), the following conditions are equivalent to the conditions on $\partial \Sigma$.

\textbf{Conditions on $\Sigma$}
\begin{outline}
    \1 The reflection symmetries through two distinct planes.
    \1 The reflection symmetry through a plane with an additional condition.
        \2 The reflection plane $\Pi$ does not meet $\partial \Sigma$.
        \2 The reflection plane $\Pi$ intersects $\partial \Sigma$ and the two components of $\partial \Sigma$ are congruent.
    \1  The rotoreflection symmetry by $A:=B\circ R\in O(3)$, where $R$ is the reflection through a plane $\Pi$ with $R(f_1)=f_2$, and $B$ is a rotation about the axis perpendicular to $\Pi$.
        \2 $B$ is an irrational rotation.
        \2 $B$ is a rotation by an angle $\theta:=\frac{b}{a}\cdot \pi$, where $a$ is an odd number and $a$ and $b$ are relatively prime. 
\end{outline}

\bibliographystyle{amsplain}
\bibliography{annot}

\end{document}